\documentclass[12pt]{article}
\usepackage{epstopdf}
\usepackage{amsmath}
\usepackage{mathrsfs}
\usepackage{amsfonts}
\usepackage{bm}
\usepackage{amssymb}
\usepackage{amsthm}
\usepackage{stmaryrd}
\usepackage{booktabs}
\usepackage{color}
\usepackage{multirow}
\usepackage{graphicx}
\usepackage{subfigure}
\usepackage{float}
\usepackage{appendix}
\usepackage{tikz}
\usetikzlibrary{calc}
\usetikzlibrary{matrix}
\usepackage{mathtools}
\usepackage{algorithm,algpseudocode, float}
\usepackage{lipsum}


\usepackage{lineno}
\usepackage{srcltx,graphicx}

\newtheorem{thm}{Theorem}[section]

\newtheorem{lemma}[thm]{Lemma}

\newtheorem{rem}[thm]{Remark}
\newtheorem{example}[thm]{Example}

\newcommand{\V}[1]{\mbox{\boldmath $ #1 $}}

\newcommand{\tr}[1]{\text{tr} #1}

\newcommand{\M}[1]{\mathbb{M} #1}
\newcommand{\FMF}[1]{\left( (F_K')^T\mathbb{M}_K F_K'\right) #1}
\newcommand{\dFMF}[1]{\det\left( (F_K')^T\mathbb{M}_K F_K'\right) #1}
\newcommand{\FMFinv}[1]{\left(\left( (F_K')^T\mathbb{M}_K F_K'\right)^{-1}\right) #1}
\newcommand{\T}[1]{\mathcal{T} #1}

\theoremstyle{definition}

\def \x{\bm{x}}

\def \J{\mathbb{J}}
\def \R{\mathbb{R}}
\theoremstyle{definition}

\newcommand{\bey}{\begin{eqnarray}}
\newcommand{\eey}{\end{eqnarray}}

\newcommand{\beq}{\begin{equation}}
\newcommand{\eeq}{\end{equation}}
\usepackage{hyperref}

\numberwithin{equation}{section} \topmargin=-2cm \oddsidemargin=1cm
\begin{document}
\baselineskip=2pc

\title{\textbf{A unifying moving mesh method for curves, surfaces, and domains based on mesh equidistribution and alignment}}

 \author{
 Min Zhang\footnote{National Engineering Laboratory for Big Data Analysis and Applications, Peking University, Beijing, 100871, China.~Chongqing Research Institute of Big Data, Peking University, Chongqing, 401121, China.  Email: minzhang@pku.edu.cn.}
 ~and~
 Weizhang Huang\footnote{Department of Mathematics, University of Kansas, Lawrence, Kansas 66045, USA. E-mail: whuang@ku.edu.}
 }

\date{}
\maketitle
\begin{abstract}
A unifying moving mesh method is developed for general $m$-dimensional geometric objects
in $d$-dimensions ($d \ge 1$ and $1\le m \le d$) including curves, surfaces, and domains.
The method is based on mesh equidistribution and alignment and
does not require the availability of an analytical parametric representation of the underlying geometric object.
Mathematical characterizations of shape and size of $m$-simplexes and properties of corresponding edge matrices
and affine mappings are derived. The equidistribution and alignment conditions
are presented in a unifying form for $m$-simplicial meshes.
The equation for mesh movement is defined based on the moving mesh PDE approach,
and suitable projection of the nodal mesh velocities is employed to ensure the mesh points stay on the underlying geometric object.
The analytical expression for the mesh velocities is obtained in a compact matrix form.
The nonsingularity of moving meshes is proved.
Numerical results for curves ($m=1$) and surfaces ($m=2$) in two and three dimensions
are presented to demonstrate the ability of the developed method to move mesh points without causing singularity
and control their concentration.
\end{abstract}

\noindent\textbf{The 2020 Mathematics Subject Classification:} 65M50, 65N50

\vspace{5pt}


\normalsize \vskip 0.2in
\newpage

\section{Introduction}
\label{sec:intr}

We are interested in the development of a unifying moving mesh method that can be used
to move simplicial meshes on a general bounded $m$-dimensional geometric object $S$ in $\R^d$ ($d \ge 1$ and $1\le m \le d$)
with or without an analytical parametric representation. The so-developed method is useful for
the computation of the evolution of $S$ and the numerical solution of partial differential equations (PDEs) defined on $S$.
Notice that $S$ can be a domain ($m = d$), a curve ($m=1 < d$), or a surface ($m=2 < d$) in $\mathbb{R}^d$.
Moreover, the method does not require the availability of an analytical parametric representation of $S$.
Generally speaking, it needs to use the information of the normal/tangent vector to ensure the mesh points stay on $S$.
The curvature of $S$ (for $m<d$) is also needed if we want to control mesh concentration based on the curvature.
These information can be obtained from a mesh that presents $S$ reasonably accurately.
Mesh examples are shown in Fig.~\ref{Fig:exd3} for a torus curve and a torus surface. The initial meshes are nonuniform, generated
with random perturbations to the location of mesh points.
The final meshes are obtained by applying the unifying moving mesh method to be presented in this work with the geometric information of normal/tangent to $S$ computed from the initial meshes and with a mesh concentration control in an attempt to make the mesh more uniform.
 \begin{figure}[H]
 \centering
\subfigure[\scriptsize{Torus curve: initial nonuniform mesh}]{
\includegraphics[width=0.36\textwidth, trim=80 0 90 30,clip]{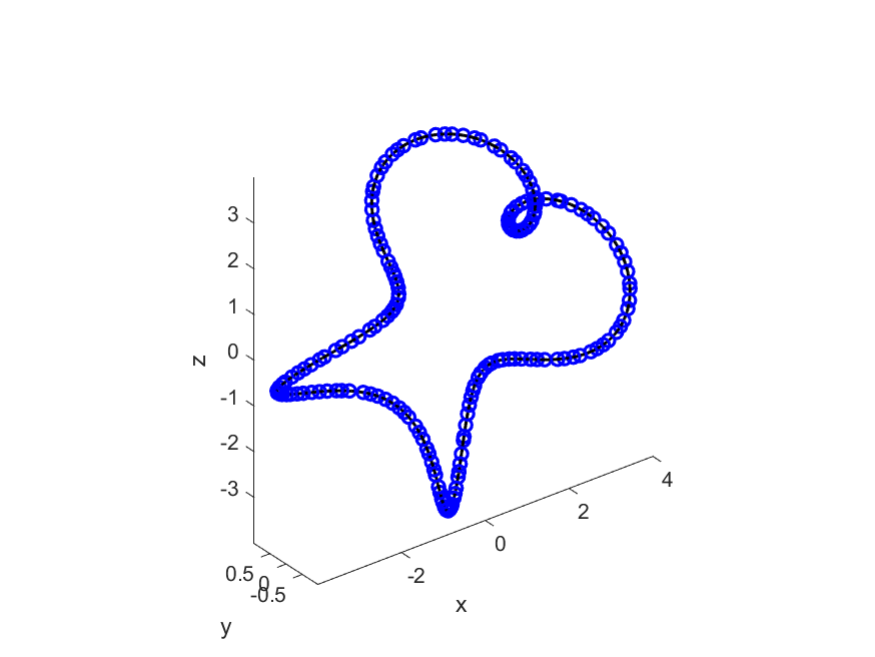}}
\hspace{1cm}
\subfigure[\scriptsize{Torus curve: final uniform mesh}]{
\includegraphics[width=0.36\textwidth, trim=80 0 90 30,clip]{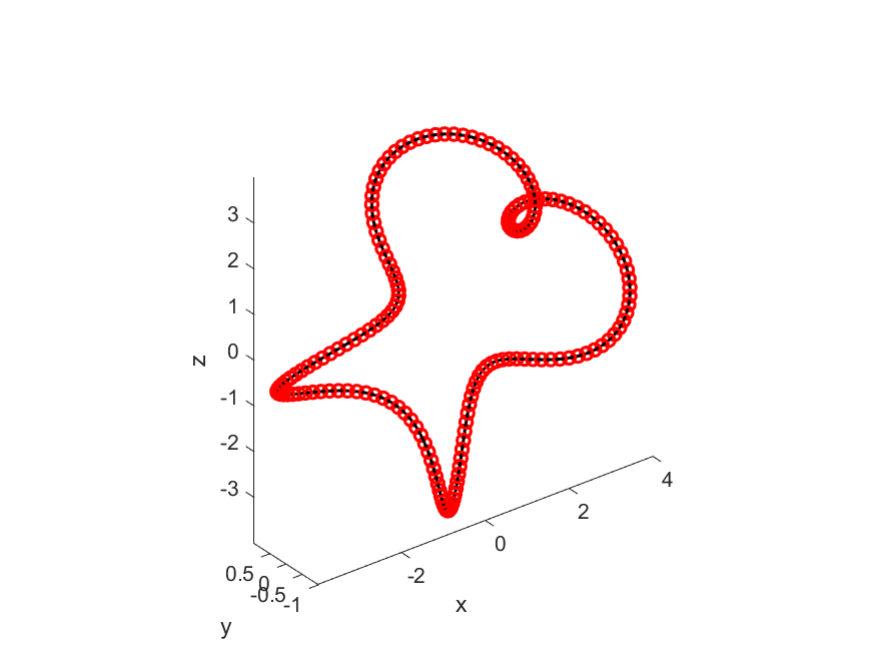}}
 \subfigure[\scriptsize{Torus surface: initial nonuniform mesh}]{
 \includegraphics[width=0.36\textwidth, trim=20 20 20 20,clip]{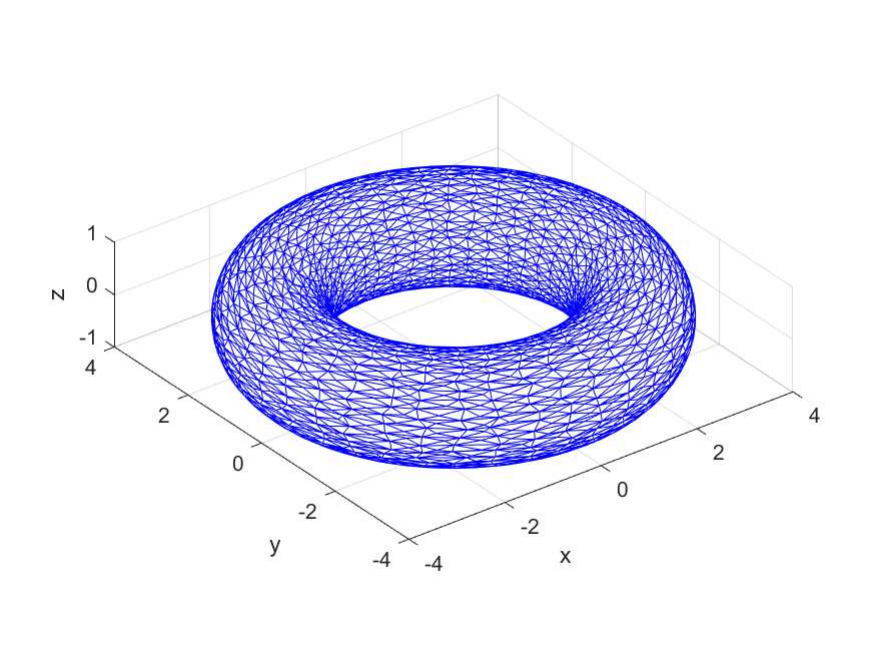}}
 \hspace{1cm}
 \subfigure[\scriptsize{Torus surface: final uniform mesh}]{
 \includegraphics[width=0.36\textwidth, trim=20 20 20 20,clip]{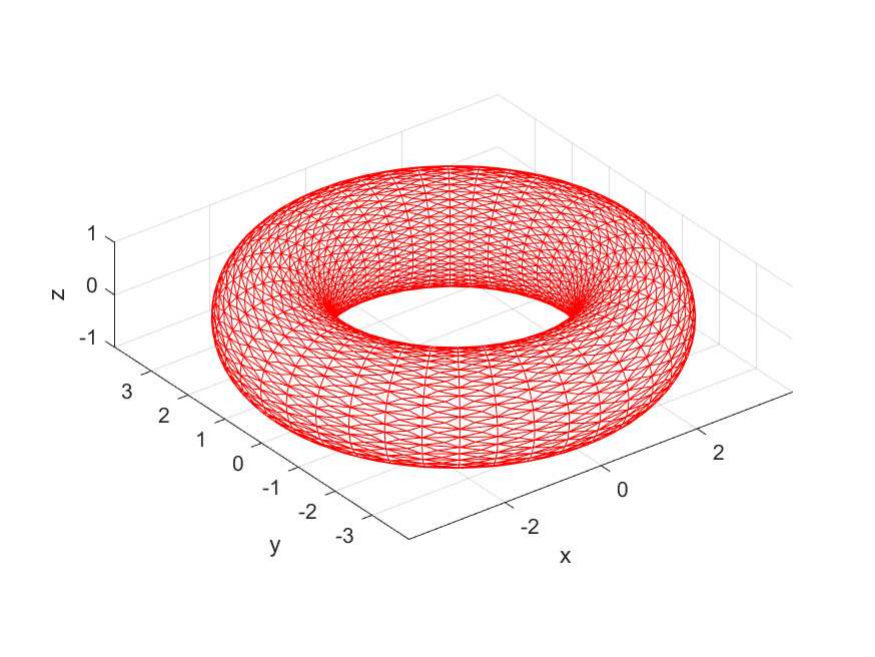}}
 \caption{
 Mesh examples for a torus curve and a torus surface in $\mathbb{R}^3$.
  }
 \label{Fig:exd3}
 \end{figure}

In the past, (adaptive) mesh movement has been studied extensively for domains ($m = d$) and various methods have been developed.
For example, Winslow \cite{W1981} proposed an equipotential method based on variable diffusion.
Brackbill and Saltzman \cite{BS1982} developed a method by combining mesh concentration, smoothness, and orthogonality.
Dvinsky \cite{D1991} proposed to use of harmonic mappings for mesh adaptation.
Huang et al. \cite{HRR1994} proposed a moving mesh PDE (MMPDE) approach
of the moving mesh method based on the equidistribution principle.
Knupp \cite{K1996} developed a functional based on the idea of conditioning the Jacobian matrix of the coordinate transformation.
Huang \cite{H2001} proposed to use mesh equidistribution and isotropy (i.e., alignment) to develop meshing functionals for
variational mesh adaptation (also see Huang and Russell \cite{HR2011-book}).
Li et al. \cite{LTZ2001} developed a moving mesh method based on harmonic maps in the framework of the Hamilton-Schoen-Yau theory \cite{D1991}.
Huang and  Kamenski \cite{HK2015}  presented a simple direct discretization for functionals used in the variational mesh generation
and adaptation. Zhang et al. \cite{M-RTE2020-cicp} developed a matrix-intersection-based MMPDE moving mesh method for radiative transfer equations.
The reader is referred to books and review articles \cite{Baines1994,BHJ2011,BHR2009,K1996,AvaryH2018,LTZ2002,HTang2006,TT2003,Tang2005,TWM1985-book,M-RTE2020,M-SWEs2021,M-SWEs2022} for moving mesh methods of domains.

On the other hand, studies of (adaptive) mesh movement for curves and surfaces are much more limited.
For example, Crestel et al. \cite{CRR} presented a moving mesh method for parametric surfaces by generalizing Winslow's meshing functional.
Browne et al. \cite{BBCW} developed a mesh adaptation method on the sphere by using a Monge-Amp\'ere type equation on the sphere.
Di et al. \cite{DLTZ2006-shpere} used a perturbed harmonic mapping moving mesh method for solving singular problems on a sphere.
MacDonald~et~al.~\cite{MMNI} developed a moving mesh method for the numerical simulation of coupled bulk-surface reaction-diffusion
equations on an evolving two-dimensional domain.
Kolasinski and Huang \cite{AvaryH2020} presented a surface-moving mesh method based on equidistribution and alignment
for general surfaces with or without explicit parameterization.
Mackenzie et al. \cite{Mackenzie-curve2019} proposed an adaptive moving mesh method for the numerical solution of a forced curve shortening geometric evolution equation.

The objective of this work is to develop a unifying moving mesh method
based on mesh equidistribution and alignment that works for general bounded $m$-dimensional geometric
objects in $\mathbb{R}^d$ ($ d \ge 1$ and $1\le m \le d$) without major modifications in its formulation.
These geometric objects include domains, surfaces, and curves. It is interesting to point out
that the unifying method reduces to the moving mesh method of \cite{H2001,HR2011-book} for bulk meshes
when $m=d$ and the surface moving mesh method of \cite{AvaryH2020} when $m = d-1$ for $d \ge 2$.
In other words, those existing bulk and surface moving mesh methods are special cases
of the unifying method. Notice that the formulations of those methods do not apply
to other cases with $ 1 \le m \le d-2$ (for $d \ge 3$) and particularly the case with curves
in three dimensions ($m = 1$ and $d = 3$).
Thus, the unifying method is useful since it covers cases that are not covered by the existing methods.
Moreover, the development of the unifying formulation can significantly simplify the theory and computer
implementation of the moving mesh method since we do not need to analyze and implement the method for each
of those cases separately.

It should be emphasized that the unifying moving mesh method does not require the availability of
an analytical parametric representation of the underlying geometry.
For the surface and curve mesh movement, the method utilizes the surface normal vectors and
curve tangent vectors to ensure that the mesh points stay on the surface and curve, respectively.
The surface normal vector (or curve tangent vector) can be computed from a mesh that represents the surface (or the curve) reasonably accurately.
Properties of edge matrices and affine mappings of $m$-simplicial elements are derived and used to characterize their size and shape.
The analytical expression for the mesh velocities is obtained in a compact matrix form.
The mesh nonsingularity of the unifying moving mesh equations is analyzed.
A selection of examples for curve mesh movement in two and three dimensions and surface mesh movement in three dimensions
are presented.

It is worth pointing out that the unifying moving mesh method is developed with the aim that it provides a mechanism for effective control of mesh concentration through a metric tensor that contains the information for the shape, size, and orientation of mesh elements.
This mesh concentration control includes the case of mesh adaptation where the metric tensor is
computed based on certain errors in some numerical approximation to the solution of some physical PDEs.
The discussion of mesh adaptation will involve physical PDEs and their discretization, which will make the current paper very long.
Moreover, mesh adaptation has been extensively studied with the bulk moving mesh method (a special case of the current method);
e.g., see  \cite{M-RTE2020,M-SWEs2021,M-SWEs2022}). For these reasons, we will not discuss mesh adaptation in this work.
Instead, we focus on the development of the unifying formulation and
the demonstration of the method's ability to move mesh points without causing singularity and control their concentration according to the uniform-mesh and curvature-based metric tensors.

This paper is organized as follows.
Properties of simplicial meshes and edge matrices are studied in \S\ref{sec:jsim}.
The equidistribution and alignment conditions for general nonuniform, simplicial bulk, surface, or curve meshes in a unifying form are presented in \S\ref{sec:M-uniform}.
The unifying moving mesh equations are established in \S\ref{sec:MMeqn}.
The theoretical analysis on the nonsingularity of the mesh trajectory for the unifying moving mesh equations is given in \S\ref{sec:mesh-non}.
Numerical examples for curves and surfaces in two and three dimensions are presented in \S\ref{sec:numerical}.
Conclusions and further remarks are given in \S\ref{sec:conclusion}.

\section{Simplicial meshes and edge matrices}
\label{sec:jsim}

Consider a bounded, connected $m$-dimensional geometric object $S$ in $\mathbb{R}^d$, where $1\le m \le d$ and $d\ge 1$.
This object can be a curve, a surface, or a domain in $\mathbb{R}^d$.
Assume that a simplicial mesh $\mathcal{T}_h$ is given on $S$ or approximating $S$.
Let $N$ and $N_v$ be the number of elements and vertices of the mesh, respectively.
In this section, we discuss the characterization of the shape and size of $m$-dimensional simplexes in $\mathbb{R}^d$
in terms of edge matrices and affine mappings.

An $m$-dimensional simplex ($m$-simplex) $K\subset\mathbb{R}^d$ is defined as the convex hull of $(m+1)$ vertices $\bm{x}_j^K\in\mathbb{R}^d,~j=0,...,m$, i.e.,
\begin{equation}\label{chull-K}
\begin{split}
K &= \text{CHull}\left(\bm{x}_0^K,\bm{x}_1^K,...,\bm{x}_m^K\right)
=\left\{\sum_{j=0}^m\lambda_j\bm{x}_j^K:~ 0<\lambda_j<1,~\sum_{j=0}^m\lambda_j=1\right \}.
\end{split}
\end{equation}
Notice that a $0$-simplex is a point, a $1$-simplex is a line segment, a $2$-simplex is a triangle, and a $3$-simplex is a tetrahedron.
The properties of $m$-simplexes that are needed in developing algorithms for mesh generation and adaptation will be described below.

Assume a reference simplicial element $\hat{K}\subset\R^m$ has been chosen. Denote its vertices
by $\bm{\xi}_j~\in \mathbb{R}^{m},~j=0,1,...,m$.
The edge matrices $E_K$ and $\hat{E}$ for $K$ and $\hat{K}$, respectively, are defined as
\begin{equation}\label{Emat}
\begin{split}
E_K &= \left[\bm{x}_{1}^K-\bm{x}_0^K, ..., \bm{x}_m^K-\bm{x}_0^K\right],\quad \x_j^K\in\mathbb{R}^d,
\\
\hat{E}& = \left[\bm{\xi}_{1} ~- ~\bm{\xi}_0~, ~...,\bm{\xi}_m ~- \bm{\xi}_0 \right],\quad ~\bm{\xi}_j\in \mathbb{R}^{m}.
\end{split}
\end{equation}
Note that $\hat{E}$ is an $m\times m$ square matrix and its inverse exists since $\hat{K}$ is not degenerate.
On the other hand, $E_K$ has the size $d\times m$ and is not square except for the case $m = d$.

Let $F_K:\hat{K} \to K $ be the affine mapping between $\hat{K}\subset \mathbb{R}^{m}$ and
$K \subset \mathbb{R}^{d}$. Then,
\begin{align*}
&\bm{x}_j^K = F_K(\bm{\xi}_j), \quad j =0,1,..., m,
\end{align*}
or
\begin{equation*}
\bm{x}_j^K-\bm{x}_0^K = F_K'\left(\bm{\xi}_j - \bm{\xi}_0\right), \quad j = 1,...,m
\end{equation*}
where $F_K'$ is the Jacobian matrix of $F_K$.
It is not difficult to see that $F_K' = E_K\hat{E}^{-1}$, which is not square except for the case $m = d$.

Edge matrices $E_K$ and $\hat{E}$ and affine mapping $F_K$ are used to characterize the size and shape of $m$-simplicial elements.

\subsection{Properties of edge matrices}
\label{sec:propE}

Without causing confusion, we call the $m$-dimensional measure of $K$ as the size of $K$, denoted by $|K|$.
Notice that the size of $K$ can be the volume (e.g., for $m = d = 3$), area (e.g., for $m = 2$ and $d=3$),
and length (e.g., for $m = 1$ and $d=3$). Denoted the diameter, minimum height, and the in-diameter
of $K$ by $h_K$, $a_K$, and $\rho_K$, respectively.
In this subsection, we study the properties of edge matrices $E_K$ and $\hat{E}$.

\begin{lemma}
\label{volK-E}
For an $m$-simplex $K\subset\mathbb{R}^d$, the size $|K|$ for $K$ can be expressed as
\begin{equation}
\label{|K|}
|K| =\frac{1}{m!}\det\left(\left(E_K^TE_K\right)^{\frac{1}{2}}\right),
\end{equation}
where $\det (\cdot )$ denotes the determinant of a matrix.
\end{lemma}

\begin{proof}
The QR decomposition of $E_K \in \mathbb{R}^{d\times m}$ reads as
\begin{equation}
\begin{split}
E_K=Q_K
\begin{bmatrix*}[c]
R_K\\
\mathbf{0}
\end{bmatrix*},
\end{split}
\end{equation}
where $Q_K \in \mathbb{R}^{d\times d}$ is an orthogonal matrix, $R_K\in\mathbb{R}^{m\times m}$ is an upper triangular matrix,
and $\mathbf{0}$ is a zero-matrix $\in \mathbb{R}^{ (d-m)\times m}$.
This indicates that $K$ is formed by rotating the convex hull with the columns vectors of {\small{$\begin{bmatrix*}[c] R_K\\ \mathbf{0}\end{bmatrix*}$}} being its edge vectors.
Since rotation does not change the size of geometric objects, $K$ has the same size as the convex hull associated with
{\small{$\begin{bmatrix*}[c] R_K\\ \mathbf{0}\end{bmatrix*}$}}, with the latter lying on the $x^{(1)}-...-x^{(m)}$ -- plane
and having the same size as the convex hull formed by the column vectors of $R_K$ in $m$-dimensions. Then,
\[
\begin{split}
|K|=\text{size}(R_K)
=\frac{1}{m!}\det(R_K)
=\frac{1}{m!}\det(R_K^TR_K)^{\frac{1}{2}}.
\end{split}
\]
Moreover, we have
\[
\begin{split}
\det\left(E_K^TE_K\right)^{\frac 1 2}
&=\det\left(
\begin{bmatrix*}[c]
R_K \\
\mathbf{0}
\end{bmatrix*}^T Q_K^TQ_K
\begin{bmatrix*}[c]
R_K\\
\mathbf{0}
\end{bmatrix*}
\right)^{\frac 1 2}
\\&=\det\left(
\begin{bmatrix*}[c]
R_K^T &\mathbf{0}
\end{bmatrix*}
\begin{bmatrix*}[c]
R_K\\
\mathbf{0}
\end{bmatrix*}\right)^{\frac 1 2}
=\det\left(R_K^TR_K\right)^{\frac 1 2} .
\end{split}
\]
Combining the above results we obtain (\ref{|K|}).
\end{proof}

\vspace{8pt}

Edge matrices can be used to compute the gradient of Lagrange-type linear basis functions that are needed in the computation of the gradient of mesh energy functions (cf. \S\ref{sec:ana-V}).

\begin{lemma}
\label{lem-dphi}
Denote the Lagrange-type linear basis function associated with the vertex $\V x_j^K\in\mathbb{R}^d$ by
$\phi_j^K$ ($j = 0, ..., m$). Then, there holds
\begin{equation}\label{dphi}
\begin{split}
\begin{bmatrix*}
\frac{\partial \phi_{1}^{K}}{\partial \x}\\
 \vdots \\
 \frac{\partial \phi_{m}^{K}}{\partial \x}
\end{bmatrix*}
=(E_K^TE_K)^{-1}E_K^T,
\quad
\frac{\partial \phi_{0}^{K}}{\partial \x} = -\sum_{j=1}^m\frac{\partial \phi_{j}^{K}}{\partial \x},
\end{split}
\end{equation}
where $\frac{\partial \phi_{j}^{K}}{\partial \x}$ is a row vector of size $1$-by-$d$.
\end{lemma}

\begin{proof}
Notice that the  linear basis functions satisfy
\begin{equation}\label{linear}
\sum_{j=0}^m\phi_{j}^{K}=1
\quad \text{and}\quad \sum_{j=0}^m \x_j^K \phi_{j}^{K}=\x.
\end{equation}
Differentiating the first equation of \eqref{linear} with respect to $\V x$, we have
\[
\frac{\partial \phi_{0}^{K}}{\partial \x} = -\sum_{j=1}^m\frac{\partial \phi_{j}^{K}}{\partial \x},
\]
which gives the second equation of (\ref{dphi}).

From \eqref{linear}, we obtain
\[
\V x-\V x_{0}^K=\sum_{j=1}^m (\V x_{j}^K-\V x_{0}^K)\phi_{j}^{K}.
\]
Differentiating the above equation with respect to $\V x$ gives
\begin{equation}\label{Ephi}
\sum_{j=1}^m(\V x_{j}^K-\V x_{0}^K)\frac{\partial \phi_{i}^{K}}{\partial \V x}=\mathbb{I}_{d\times d},
\end{equation}
which can be written as
\begin{equation}\label{Ephi-1}
\left[ \x_{1}^K-\x_{0}^K,...,\x_{m}^K-\x_{0}^K \right]
\begin{bmatrix*}
\frac{\partial \phi_{1}^{K}}{\partial \x}\\
 \vdots \\
 \frac{\partial \phi_{m}^{K}}{\partial \x}
\end{bmatrix*}=
E_K\begin{bmatrix*}
\frac{\partial \phi_{1}^{K}}{\partial \x}\\
 \vdots \\
 \frac{\partial \phi_{m}^{K}}{\partial \x}
\end{bmatrix*}=\mathbb{I}_{d\times d}.
\end{equation}
Thus, we have
\[
E_K^TE_K
\begin{bmatrix*}
\frac{\partial \phi_{1}^{K}}{\partial \x}\\
 \vdots \\
 \frac{\partial \phi_{m}^{K}}{\partial \x}
\end{bmatrix*}
=E_K^T\quad\hbox{or}\quad
\begin{bmatrix*}
\frac{\partial \phi_{1}^{K}}{\partial \x}\\
 \vdots \\
 \frac{\partial \phi_{m}^{K}}{\partial \x}
\end{bmatrix*}
=(E_K^TE_K)^{-1}E_K^T,
\]
which gives the first equation of (\ref{dphi}).
\end{proof}

\vspace{8pt}

Notice that $(E_K^TE_K)^{-1}E_K^T$ is the Moore-Penrose pseudo-inverse of $E_K$, i.e.,
\begin{equation}\label{ginvEK}
E_K^{+} = (E_K^TE_K)^{-1}E_K^T .
\end{equation}
We introduce the $\bm{q}$-vectors as
\begin{equation}
\label{qv}
(E_K^{+})^T = [\bm{q}_1^K,...,\bm{q}_m^K],\quad \bm{q}_0^K = - \sum_{j=1}^{m} \bm{q}_j^K.
\end{equation}
Then, Lemma~\ref{lem-dphi} implies
\[
\bm{q}_j^K = \left ( \frac{\partial \phi_{j}^{K}}{\partial \x}\right )^T, \quad j = 0, ..., m.
\]
From this and the fact that the $j$-th facet $F_j^K$ of $K$ (the facet formed by all vertices except $\bm{x}_j^K$)
is a contour plane of $\phi_{j}^{K}$, we can obtain the following lemma.

\begin{lemma}
\label{lem-q}
For $j=0,1,..m$,
\begin{enumerate}
  \item[(i)] Vector $\bm{q}_j^K$ is an inward normal to facet $F_j^K$.
  \item[(ii)] The $j$-th height of $K$ (the distance from vertex $\bm{x}_j^K$ to facet $F_j^K$) is equal
  to $1/\|\bm{q}_j^K\|$, i.e., $a_j^K = 1/\|\bm{q}_j^K\|$.
\end{enumerate}
\end{lemma}

\begin{lemma}
\label{lem-hKaK}
For any $m$-simplex $K\subset\mathbb{R}^d$, there holds
\begin{align}
& \frac{h_K^2}{m}\leq \| E_K^TE_K\| \leq m h_K^2,  \label{hk}\\
&\frac{1}{m^2a_K^2}\leq \| (E_K^TE_K)^{-1}\| \leq \frac{m}{a_K^2}, \label{aK}
\end{align}
where $a_K$ is the minimum height of $K$, i.e., $a_K = \min\limits_{0\le j \le m} a_j^K$.
\end{lemma}

\begin{proof}
It is not difficult to show
\begin{equation}
\frac{1}{m}\tr{\left(E_K^TE_K\right)} \leq    \| E_K^TE_K\|  \leq \tr{\left(E_K^TE_K\right)} .
\end{equation}
Moreover, by direct calculation, we have
\[
\tr{\left(E_K^TE_K\right)} = \sum_{j=1}^m (\V x_j^K-\V x_0^K)^T (\V x_j^K-\V x_0^K) ,
\]
which implies
\begin{equation}
h^2_K \leq \tr{\left(E_K^TE_K\right)} \leq m h^2_K.
\end{equation}
Then, we have obtained \eqref{hk}.

Similarly, we have
\begin{equation}\label{eig-tr-inv}
\frac{1}{m}\tr{\left((E_K^TE_K)^{-1}\right)} \leq    \| (E_K^TE_K)^{-1}\|  \leq \tr{\left((E_K^TE_K)^{-1}\right)} .
\end{equation}
From the definition of $\bm{q}$-vectors and Lemma \ref{lem-q}, we have
\begin{equation}\label{ak-r}
 \frac{1}{\min\limits_{1\leq j\leq m} (a_j^K)^2}\leq \tr{\left((E_K^TE_K)^{-1}\right)} = \tr{\left( E_K^{+}(E_K^{+})^T\right)} =\sum_{j=1}^m \|\bm{q}_j^K\|^2 \leq \frac{m}{ a^2_K}.
\end{equation}
Then, we have
\begin{equation}\label{ak-r1m}
 \frac{1}{m \min\limits_{1\leq j\leq m} (a_j^K)^2}\leq \| (E_K^TE_K)^{-1}\| \leq \frac{m}{ a^2_K}.
\end{equation}
Furthermore, using the triangle inequality and Cauchy-Schwarz inequality we have
\begin{equation}
\|\bm{q}_0^K\|^2 =\|\sum_{j=1}^m \bm{q}_j^K\|^2 \leq  m \sum_{j=1}^m \| \bm{q}_j^K\|^2= m ~\tr{\left((E_K^TE_K)^{-1}\right)}\leq m^2 ~\| (E_K^TE_K)^{-1}\|.
\end{equation}
and
\begin{equation}\label{ak-r0}
\frac{1}{m^2 (a_0^K)^2} \leq \| (E_K^TE_K)^{-1}\|
\leq \frac{m}{ a^2_K}.
\end{equation}
Hence, we can obtain \eqref{aK} by combined \eqref{ak-r1m} and \eqref{ak-r0}.
\end{proof}

\subsection{Properties of affine mapping}

\begin{lemma}
\label{lem:volK}
For any $m$-simplexes $K\subset \mathbb{R}^{d}$ and $\hat{K} \subset \mathbb{R}^{m}$,
there holds
\begin{equation}\label{area-1}
\frac{|K|}{|\hat{K}|} =\det\left(\left(F_K'\right)^TF_K'\right)^{\frac{1}{2}} .
\end{equation}
\end{lemma}

\begin{proof}
Recalling $F_K' = E_K\hat{E}^{-1}$, we have
\begin{align*}
\det\left(\left(F_K'\right)^TF_K'\right)^{\frac{1}{2}}&
=\det\left(\hat{E}^{-T}E_K^TE_K\hat{E}^{-1}\right)^{\frac{1}{2}}\\
&=\det\left(\hat{E}^{-T}\right)^{\frac{1}{2}}\det\left(E_K^TE_K\right)^{\frac{1}{2}}
\det\left(\hat{E}^{-1}\right)^{\frac{1}{2}}\\
&=\det(\hat{E})^{-1}\det\left(E_K^TE_K\right)^{\frac{1}{2}}\\
&=\frac{1}{m!\, |\hat{K}|}\det\left(E_K^TE_K\right)^{\frac{1}{2}}.
\end{align*}

Therefore, 
\begin{equation*}
\begin{split}
\det\left(\left(F_K'\right)^TF_K'\right)^{\frac{1}{2}}
=\frac{1}{m!\, |\hat{K}|}\det\left(E_K^TE_K\right)^{\frac{1}{2}}
= \frac{|K|}{|\hat{K}|} .
\end{split}
\end{equation*}

\end{proof}

\begin{lemma}
\label{lem-sEK}
For any $m$-simplexes $K \subset \mathbb{R}^d$ and $\hat{K}\subset \mathbb{R}^m$,  we have
\begin{align}
& \frac{h_K^2}{m \hat{h}^2}\leq \big\| (F'_K)^TF'_K\big\| \leq \frac{m h_K^2}{\hat{a}^2},
\label{s-1}\\
&\frac{\hat{a}^2}{m^2a_K^2}\leq \big\|\left ((F'_K)^TF'_K\right)^{-1}\big\| \leq \frac{m\hat{h}^2}{a_K^2},\label{s-2}
\end{align}
where $\hat{h}$ and $\hat{a}$ are the diameter and minimum height of $\hat{K}$.
\end{lemma}

\begin{proof}
The proof is similar to that of Lemma~\ref{lem-hKaK}.
\end{proof}

\vspace{8pt}

We now consider the mathematical characterization of similarity among simplexes regarding affine mappings.
The simplex similarity is one of the two properties defining uniform meshes.

\begin{lemma}
\label{lem:Ksim}
Any two $m$-simplexes $\hat{K}\subset\R^m$ and $K\subset\R^d$ are similar if and only if one of the following conditions holds.
\begin{enumerate}
\item[(i)] $K$ can be obtained from $\hat{K}$ by a sequence of rotation, translation, and dilation transformations.
  In this case, $\hat{K}$ is considered as an $m$-simplex embedded in $\R^d$.
\item[(ii)] The Jacobian matrix $F'_K$ can be expressed as a product of scalar and orthogonal matrices (cf. (\ref{FUIV-1})).
\item[(iii)] There exists a constant $\theta_K$ such that
{\em
     \begin{equation}\label{Keq}
    \left(F_K'\right)^TF_K'=\theta_K\mathbb{I}.
    \end{equation}
    }
\item[(iv)] $F'_K$ satisfies
{\em
   \begin{equation}
   \label{Ksim}
    \frac{1}{m}\tr\left(\left(F_K'\right)^TF_K'\right)
    =\det\left(\left(F_K'\right)^TF_K'\right)^{\frac{1}{m}},
    \end{equation}
    }
    where $\text{tr}(\cdot)$ denotes the trace of a matrix.
\item[(v)] $F'_K$ satisfies
{\em
   \begin{equation}
   \label{Ksim-2}
    \frac{1}{m}\tr\left(\left((F_K')^TF_K'\right)^{-1}\right)
    =\det\left(\left((F_K')^TF_K'\right)^{-1}\right)^{\frac{1}{m}}.
    \end{equation}
    }
\end{enumerate}
\end{lemma}

\begin{proof}
(i) is obvious in geometry. From the singular value decomposition and (i),
$F_K'\in \mathbb{R}^{d\times m}$ can be expressed as
\begin{equation}
F_K' = \sqrt{\theta_K} \, U \begin{bmatrix*}[c]
\mathbb{I}\\
\mathbf{0}
\end{bmatrix*} V^T,
\label{FUIV-1}
\end{equation}
where $\theta_K$ is a constant representing dilation, $U \in \mathbb{R}^{d\times d}$ and $V\in\mathbb{R}^{m\times m}$
are orthogonal matrices representing rotation. Thus, we obtain (ii). Notice that this is equivalent to
\begin{equation*}
(F_K')^T F_K' = \theta_K^2 \mathbb{I},
\end{equation*}
which gives (iii). It also states that the eigenvalues of $(F_K')^T F_K'$ are equal.
(iv) and (v) are equivalent statements of the equality of the eigenvalues.
\end{proof}

\subsection{Simplexes in Riemannian metric}

In this subsection, we extend the results in the previous subsection to the situation
with the Riemannian metric prescribed by a metric tensor $\mathbb{M}=\mathbb{M}(\bm{x})$ defined on $S$.

We consider a symmetric and uniformly positive definite metric tensor $\mathbb{M}(\bm{x})$ which satisfies
\begin{equation}\label{Mbound}
\underline{\rho} \mathbb{I}\le\mathbb{M}(\bm{x}) \le \overline{\rho} \mathbb{I}, \quad \forall \bm{x}\in S\subset \mathbb{R}^d
\end{equation}
where $\underline{\rho}$ and $\overline{\rho}$ are two positive constants and
the less-than-or-equal-to sign is in the sense of negative semi-definiteness.  We define
the average of $\mathbb{M}$ over $K$ as
\[
\mathbb{M}_K=\frac{1}{|K|}\int_{K}\mathbb{M}(\bm{x})d\bm{x}.
\]
Recall that the distance measure in the Riemannian metric $\mathbb{M}_K$ is defined as
\begin{equation}\label{Mnorm}
\|\bm{x}\|_{\mathbb{M}_K}=\sqrt{\bm{x}^T\mathbb{M}_K\bm{x}}
=\sqrt{\left(\mathbb{M}_K^{\frac{1}{2}}\bm{x}\right)^T
\left(\mathbb{M}_K^{\frac{1}{2}}\bm{x}\right)}
=\left\|\mathbb{M}_K^{\frac{1}{2}}\bm{x}\right\|,
\end{equation}
where $\| \cdot \|$ denotes the standard Euclidean norm.  This implies that the geometric properties of $K$ in the metric $\mathbb{M}_K$ can be obtained from those of $\mathbb{M}_K^{\frac{1}{2}}K$ in the Euclidean metric.
For example, the edge matrix of $K$ in the metric $\mathbb{M}_K$ is given by
\begin{equation}\label{EVinM}
 E_{K,\mathbb{M}} = \mathbb{M}_K^{\frac{1}{2}} E_{K},
 \end{equation}
and the Jacobian matrix of the affine mapping from $\hat{K}\subset \mathbb{R}^{m}$ to $\mathbb{M}_K^{\frac{1}{2}}K$ is given by
\begin{equation}\label{JinM}
F'_{K,\mathbb{M}}=E_{K,\mathbb{M}}\hat{E}^{-1} =\left(\mathbb{M}_K^{\frac{1}{2}}E_K\right) \hat{E}^{-1}=\mathbb{M}_K^{\frac{1}{2}} F_K'.
\end{equation}
Note that $E_{K,\mathbb{M}}$ and $F_{K,\mathbb{M}}'$ are matrices of size $d\times m$.
Moreover, the $\bm{q}$-vectors of $K$ in the metric $\mathbb{M}_K$ are given by the pseudo-inverse $E_{K,\mathbb{M}}^{+}$, i.e.,
\begin{equation}
\label{qv-M}
\begin{split}
& [\bm{q}_{1,\mathbb{M}},...,\bm{q}_{m,\mathbb{M}}]
= (E_{K,\mathbb{M}}^{+})^T
=\mathbb{M}_K^{-\frac{1}{2}} (E_K^{+})^T
= \mathbb{M}_K^{-\frac{1}{2}} [\bm{q}_{1},...,\bm{q}_{m}],\\
& \bm{q}_{0,\mathbb{M}} = - \sum_{j=1}^{m} \bm{q}_{j,\mathbb{M}}
= - \sum_{j=1}^{m} \mathbb{M}_K^{-\frac{1}{2}}\bm{q}_{j}.
\end{split}
\end{equation}

We list some other results in the following lemmas without proof.

\begin{lemma}
\label{lem:volK-M}
Denote the size of $K$ in the metric $\mathbb{M}_K$ by $|K|_{\mathbb{M}}$. There holds
\begin{align}
&\frac{|K|_{\mathbb{M}}}{|\hat{K}|}=\det\FMF^{\frac{1}{2}},
\label{|K|-M-1}\\
& |K|_{\mathbb{M}} = \frac{1}{m!}\det\left(E_K^{T}\mathbb{M}_KE_K\right)^\frac{1}{2}.
\label{|K|-M-2}
\end{align}
\end{lemma}

\begin{lemma}
\label{lem:Ksim-M}
Any $m$-simplex $K\subset\R^d$ measured in the metric $\mathbb{M}_K$ is similar to an $m$-simplex $\hat{K}\subset\R^m$
measured in the Euclidean metric if and only if one of the following conditions holds.
\begin{enumerate}
\item[(i)] $\mathbb{M}_K^{\frac{1}{2}} K$ can be obtained from $\hat{K}$ by a sequence of rotation, translation, and dilation transformations.
  In this case, $\hat{K}$ is considered as an $m$-simplex embedded in $\R^d$.
\item[(ii)] $\mathbb{M}_K^{\frac{1}{2}} F'_K$ can be expressed as a product of scalar and orthogonal matrices.
\item[(iii)] There exists a constant $\theta_K$ such that
{\em
\begin{equation}\label{Keq-2}
    \left(F_K'\right)^T \mathbb{M}_K F_K'=\theta_K\mathbb{I}.
    \end{equation}
    }
\item[(iv)] $F'_K$ satisfies
 {\em
 \begin{equation}
   \label{Ksim-3}
    \frac{1}{m}\text{tr}\left(\left(F_K'\right)^T \mathbb{M}_K F_K'\right)
    =\det\left(\left(F_K'\right)^T \mathbb{M}_K F_K'\right)^{\frac{1}{m}},
    \end{equation}
    }
    where $\text{tr}(\cdot)$ denotes the trace of a matrix.
\item[(v)] $F'_K$ satisfies
 {\em
 \begin{equation}
   \label{Ksim-4}
    \frac{1}{m}\text{tr}\left(\left(\left(F_K'\right)^T \mathbb{M}_K F_K'\right)^{-1}\right)
    =\det\left(\left((F_K')^T \mathbb{M}_K F_K'\right)^{-1}\right)^{\frac{1}{m}}.
    \end{equation}
    }
\end{enumerate}
\end{lemma}

\begin{lemma}
\label{lem:eq-tr-norm}
For any $m$-simplexes $K \subset \mathbb{R}^d$ and $\hat{K}\subset \mathbb{R}^m$, there holds
{\em
\begin{align}
\frac{1}{m} \big\| (F'_K)^T\mathbb{M}_KF'_K\big\| &\leq
\frac{1}{m} \tr \left( (F_{K}')^T\mathbb{M}_K F_{K}'\right)
\leq \big\| (F'_K)^T\mathbb{M}_KF'_K\big\| ,\label{eq-tr}\\
\frac{1}{m} \big\| \left((F'_K)^T\mathbb{M}_KF'_K\right)^{-1}\big\| &\leq
\frac{1}{m} \tr \left(\left((F'_K)^T\mathbb{M}_KF'_K\right)^{-1}\right)
\leq \big\| \left((F'_K)^T\mathbb{M}_KF'_K\right)^{-1}\big\|.\label{eq-tr-inv}
\end{align}
}
\end{lemma}

\begin{lemma}
\label{lem:h-M}
For any $m$-simplexes $K \subset \mathbb{R}^d$ and $\hat{K}\subset \mathbb{R}^m$, there holds
\begin{align}
& \frac{h^2_{K,\mathbb{M}}}{\hat{h}^2}\leq \big\| (F'_K)^T\mathbb{M}_KF'_K\big\| \leq \frac{h^2_{K,\mathbb{M}}}{\hat{a}^2},
\label{hr-M-1}
\\
&\frac{\hat{a}^2}{a_{K,\mathbb{M}}^2}
\le \big\| \left((F'_K)^T\mathbb{M}_KF'_K\right)^{-1}\big\|
\le \frac{m^2\hat{h}^2}{a_{K,\mathbb{M}}^2},
\label{ar-M-1}
\end{align}
where $\hat{h}$ and $\hat{a}$ are the diameter and minimum height of $\hat{K}$ measured in the Euclidean metric and
$h_{K,\mathbb{M}}$ and $a_{K,\mathbb{M}}$ are the diameter and minimum height of $K$ measured in the metric $\mathbb{M}_K$.
\end{lemma}

We give a lower bound for the size of $K $ in the metric $\mathbb{M}_K$ in terms of the minimum height $a_{K,\mathbb{M}}$.

\begin{lemma}
\label{lem:volKM-aK}
For any $m$-simplex $K \subset \mathbb{R}^d$, there holds
\begin{equation}
\label{kak}
|K|_{\mathbb{M}} \ge \frac{ a_{K,\mathbb{M}}^{m}}{m^{\frac{m}{2}}m!}~ .
\end{equation}
\end{lemma}

\begin{proof}
From Lemma \ref{lem:volK-M}, we have
\begin{equation}
\label{KM-0}
\begin{split}
|K|_{\mathbb{M}} = \dfrac{1}{m!} \det\left(\left(\mathbb{M}_K^{\frac{1}{2}}E_K\right)^T
\left(\mathbb{M}_K^{\frac{1}{2}}E_K\right)\right)^{\frac{1}{2}}.
\end{split}
\end{equation}
Recall the QR-decomposition of $\mathbb{M}_K^{\frac{1}{2}}E_K$ reads as
\begin{align}\label{QR-MEK}
\mathbb{M}_K^{\frac{1}{2}}E_K = Q_K\left[\begin{matrix} R_K\\ \V 0\end{matrix}\right],
\end{align}
where $Q_K \in \mathbb{R}^{d\times d}$ is an unitary matrix, $R_K\in\mathbb{R}^{m\times m}$ is an upper triangular matrix, and $\mathbf{0}\in \mathbb{R}^{(d-m)\times m}$ is a zero-matrix. Substitute \eqref{QR-MEK} into \eqref{KM-0}, we have
\begin{align*}
|K|_{\mathbb{M}}
 &= \dfrac{1}{m!} \det\left(\begin{bmatrix*}[c]
R_K\\ \mathbf{0}
\end{bmatrix*}^TQ_K^T Q_K
\begin{bmatrix*}[c]
R_K\\ \mathbf{0}
\end{bmatrix*}\right)^{\frac{1}{2}}\\
 & = \dfrac{1}{m!} \det(R_K^TR_K)^{\frac{1}{2}}\\
 & = \dfrac{1}{m!} \prod_{i=1}^{m}\sigma_i\\
& \ge \dfrac{1}{m!m^{\frac{m}{2}}} a_{R_{K}}^{m},
 \end{align*}
where $\sigma_i$, $i = 1, ..., m$, are the singular values of $R_K$, $a_{R_K}$ is the minimum height of the simplex formed by the columns of $R_K$, and the fact \cite{BCY2018} that  $\sigma_i \ge \frac{a_{R_K}}{\sqrt{m}}$ has been used.

Since $Q_K$ is a rotation matrix, the minimum height of $K$ with respect to the metric $\mathbb{M}_K$ is the same as
the minimum height of the convex hull formed by the columns of $R_K$, i.e., $a_{K,\mathbb{M}} = a_{R_K}$.
Thus, we have obtained \eqref{kak}.
\end{proof}

\vspace{8pt}

It is instructional to see some special cases of \eqref{kak} with $d\geq m$:
\begin{itemize}
\item for $m=1$, $K$ is a line segment, $|K|_{\mathbb{M}} \ge a_{K,\mathbb{M}}$;
\item for $m=2$, $K$ is a triangle, $|K|_{\mathbb{M}} \ge \dfrac{1}{4}~a_{K,\mathbb{M}}^{2}$;
\item for $m=3$, $K$ is a tetrahedral, $|K|_{\mathbb{M}} \ge \dfrac{1}{18\sqrt{3}}~a_{K,\mathbb{M}}^{3}$.
\end{itemize}

\section{Mathematical characterization of nonuniform meshes}
\label{sec:M-uniform}

In this section, we use the properties of uniform meshes to establish the mathematical characterization of nonuniform meshes
that are viewed as uniform in some Riemannian metrics. To be more specific, we present the so-called
equidistribution and alignment conditions in a unifying formulation for an $m$-simplicial mesh
that can be a bulk, surface, or curve mesh. These conditions are used to characterize the size, shape, and orientation of the $m$-simplicial mesh and develop an energy function for mesh generation and adaptation.

\subsection{Equidistribution and alignment conditions}

We define a uniform mesh as a mesh whose elements have the same size and are similar to a reference element.
A nonuniform mesh is viewed as a uniform one in some Riemannian metrics. Consider a mesh $\mathcal{T}_h$ for $S$.
We assume that we are given metric tensor $\M = \M(\bm{x})$.
Then, the uniformity of $\mathcal{T}_h$ in metric $\M$ can be characterized
by the equidistribution and alignment conditions as follows.

The equidistribution condition requires that all of the elements in the mesh $\mathcal{T}_h$  have the same size, i.e.,
\begin{equation}\label{equ-0}
\begin{split}
|K|_{\mathbb{M}}
=\frac{\sigma_h}{N}, \quad \forall K \in \mathcal{T}_h,
\end{split}
\end{equation}
where $\sigma_h=\sum\limits_{K\in\mathcal{T}_h}|K|_{\mathbb{M}}$.
Recalling from Lemma~\ref{lem:volK-M},
\begin{equation*}
|K|_{\mathbb{M}} = |\hat{K}|\det\FMF^{\frac{1}{2}}.
\end{equation*}
the equidistribution condition becomes
\begin{equation}
\label{eq-1}
\det\FMF^{\frac{1}{2}}
=\frac{\sigma_h}{N |\hat{K}|}, \quad \forall K \in \mathcal{T}_h .
\end{equation}

Moreover, the alignment condition requires that the elements of $\mathcal{T}_h$ be similar to the reference element $\hat{K}$.
From Lemma \ref{lem:Ksim-M}, the alignment condition reads as
\begin{equation}
\label{ali-0}
\frac{1}{m}\tr\FMF
=\det\FMF^{\frac{1}{m}},
\quad
\forall K \in \mathcal{T}_h,
\end{equation}
or
\begin{equation}
\label{ali-1}
\frac{1}{m}\tr\FMFinv=\det \FMFinv^{\frac{1}{m}},\quad
\forall K \in \mathcal{T}_h.
\end{equation}

\begin{rem}
For bulk meshes ($m=d$), we have $|\hat{K}| = \frac{1}{d!} \det(\hat{E})$ and $|K| = \frac{1}{d!}\det(E_K)$. Then the equidistribution condition \eqref{eq-1} becomes
\begin{equation}
\label{equ-d}
|K|\det\left(\M_K\right)^{\frac12}
=\frac{\sigma_h}{N}, \quad \forall K \in \mathcal{T}_h.
\end{equation}
\end{rem}

\subsection{Mesh energy functions}
\label{sec:Ieq}
Now, we are ready to formulate the mesh energy function based on the equidistribution condition
(\ref{eq-1}) and the alignment condition (\ref{ali-1}).

We want to generate an ``optimal'' mesh satisfying
these conditions. We first formulate energy functions of a general mesh $\mathcal{T}_h$ (which is not assumed to
satisfy (\ref{eq-1}) and (\ref{ali-1})) and then optimize these functions with the hope that
an obtained approximate optimizer satisfies (\ref{eq-1}) and (\ref{ali-1}) more closely than the initial mesh.

We first consider the mesh energy function for the equidistribution condition (\ref{eq-1}).
To this end, we have the following lemma.
\begin{lemma}\label{eq-c}
For any  dimensionless parameter $p>1$, there holds
\begin{equation}\label{Holders}
\begin{split}
&\sum_{K\in\T_h}\frac{|K|_{\mathbb{M}}}{\sigma_h}\cdot \frac{1}{|K|_{\mathbb{M}}}\le
\left( \sum_{K\in\mathcal{T}_h} \frac{|K|_{\mathbb{M}}}{\sigma_h} \cdot \left(\frac{1}{|K|_{\mathbb{M}}}\right)^{p}\right)^{\frac{1}{p}},
\end{split}
\end{equation}
with equality if and only if
\[
|K|_{\mathbb{M}} = |\hat{K}|\det\FMF^{\frac{1}{2}}= \text{constant},\quad\quad \forall~K\in\mathcal{T}_h.
\]
\end{lemma}
\begin{proof}
This follows from H\"older's inequality.

\end{proof}

The above lemma implies that minimizing the difference between the right- and left-hand sides of (\ref{Holders})
tends to make $|K|_{\mathbb{M}}$ constant for all $K\in\mathcal{T}_h$.
Since the left-hand side is a constant, we define the equidistribution energy function as
\begin{align}
\mathcal{I}_{eq} &= m^{\frac{m p}{2}}|\hat{K}|^{p}
\sum_{K\in\mathcal{T}_h} |K|_{\mathbb{M}}\cdot \left(\frac{1}{|K|_{\mathbb{M}}}\right)^{p}
\notag \\
&=m^{\frac{m p}{2}}\sum_{K\in\T_h} |\hat{K}|\det\FMF^{\frac{1}{2}}\det\FMF^{-\frac{p}{2}}
\notag
\\
&= m^{\frac{m p}{2}}
\sum_{K\in\T_h}|\hat{K}|\det\FMFinv^{\frac{p-1}{2}},
\label{Ieq}
\end{align}
where $p>1$ is a dimensionless parameter. The factor $m^{\frac{m p}{2}}|\hat{K}|^{p}$ has been added here to match the energy function for alignment which will be discussed later.
It is interesting to point out that the above condition is the same as that in \cite{H2001} for bulk meshes ($m=d$)
and that in \cite{AvaryH2020} for surface meshes ($m=d-1$).

\vspace{8pt}

Next, we consider the mesh energy function for the alignment condition (\ref{ali-1}).
Notice that the left- and right-hand sides are the arithmetic mean and the geometric mean of the eigenvalues
of the matrix $\FMF^{-1}$, respectively. The inequality of arithmetic and geometric means gives
\begin{equation}\label{ine}
\frac{1}{m}\tr\FMFinv\ge \det \FMFinv^{\frac{1}{m}},
\end{equation}
with equality if and only if all of the eigenvalues are equal.
From this, for any $p>0$ we have
\[
\tr \FMFinv^{\frac{m p}{2}}
\geq m^{\frac{mp}{2}} \det \FMFinv^{\frac{p}{2}}.
\]
This implies that minimizing the difference between the two sides tends to make the mesh satisfy the alignment
condition (\ref{ali-1}). Multiplying the above inequality with $|K|_{\mathbb{M}}$ and summing the difference overall
elements, we obtain the alignment energy function as
\begin{align}
\mathcal{I}_{ali}&=
\sum_{K\in\mathcal{T}_h}|K|_{\mathbb{M}}\left(
\tr \FMFinv^{\frac{mp}{2}}
-m^{\frac{mp}{2}}\det \FMFinv^{\frac{p}{2}}\right)
\notag \\
&=\sum_{K\in\mathcal{T}_h}|\hat{K}|\det\FMF^{\frac{1}{2}}
\tr \FMFinv^{\frac{mp}{2}}
\notag
\\
& \qquad ~~~~-m^{\frac{mp}{2}}
\sum_{K\in\mathcal{T}_h}|\hat{K}|\det\FMF^{\frac{1}{2}}\det \FMFinv^{\frac{p}{2}}
\notag
\\
& =
\sum_{K\in\mathcal{T}_h}|\hat{K}|\det\FMF^{\frac{1}{2}}
\tr \FMFinv^{\frac{mp}{2}}
\notag
\\
& \qquad ~~~-m^{\frac{mp}{2}}
\sum_{K\in\mathcal{T}_h}|\hat{K}|\det \FMFinv^{\frac{p-1}{2}} .
\label{Iali}
\end{align}

We now formulate a mesh energy function for combined equidistribution and alignment.
One way to ensure this is to average (\ref{Ieq}) and (\ref{Iali}). Notice that
(\ref{Ieq}) and (\ref{Iali}) have the same dimension. As a consequence, we can use
a dimensionless parameter $\theta\in[0,1]$. Thus, we have
\begin{align}
\mathcal{I}_h&=\theta \mathcal{I}_{ali}+(1-\theta) \mathcal{I}_{eq}
\notag \\& =\theta \sum_{K\in\mathcal{T}_h} |\hat{K}| \dFMF^{\frac 1 2}
\tr\FMFinv^{\frac{mp}{2}}
\notag \\
& \quad  + (1-2\theta) m^{\frac{mp}{2}}
\sum_{K\in\T_h} |\hat{K}|\dFMF^{\frac 1 2}
\det\FMFinv^{\frac{p}{2}},
\label{Ih-1}
\end{align}
where $p > 1$ and $\theta \in [0,1]$ are dimensionless parameters.
It is instructional to point out that, in the case with $m = d$, the function (\ref{Ih-1}) can be viewed as
a Riemann sum of the continuous mesh energy functional developed in \cite{H2001} for bulk meshes based on equidistribution
and alignment conditions.
When $m < d$, $\FMF^{-1}$ in (\ref{Ih-1}) cannot be further simplified since $F_K^{'}$ is not a square matrix.
 We can rewrite (\ref{Ih-1}) into
 \begin{equation}
 \label{Ih-2-0}
 \mathcal{I}_h(\mathcal{T}_h)=\sum_{K\in \mathcal{T}_h} |\hat{K}|  G_K,
 \end{equation}
 where
 \begin{equation}\label{Ih-2-1}
 \begin{split}
 G_K =&~\theta~\det\FMF^{\frac{1}{2}}\tr\FMFinv^{\frac{mp}{2}}
 \\& + (1-2\theta)~ m^{\frac{mp}{2}} \det\FMFinv^{\frac{p-1}{2}}.
 \end{split}
 \end{equation}

\begin{rem}
In the numerical computation, we take $p = 3/2$ and $\theta = 1/3$ in the meshing function \eqref{Ih-2-1}.
This choice has been used and seems to work well in a variety of applications, e.g., see
\cite{AvaryH2020,M-RTE2020,M-SWEs2021,M-SWEs2022}.
\end{rem}

\section{Unifying moving mesh equations}
\label{sec:MMeqn}

In this section, we employ an MMPDE method in a unifying for to minimize the meshing function ($\ref{Ih-1}$).

First, we recall that when $S$ is a domain ($m=d$, bulk meshes), the MMPDE approach defines the moving mesh equation
as a gradient system of the mesh energy function $\mathcal{I}_h$, i.e.,
\begin{equation}\label{MMPDE-0}
 \frac{d \V{x}_i}{d t} = - \frac{P_i}{\tau} ~\left ( \frac{\partial \mathcal{I}_h}{\partial \V{x}_i} \right )^T, \quad i = 1, ..., N_v
\end{equation}
where $\frac{d \V{x}_i}{d t}$ is the nodal mesh velocity, $\tau$ is a positive parameter used for adjusting the time scale of mesh movement,
$P_i$ is a scalar function used to make the equation \eqref{MMPDE-0} invariant under the scaling transformation of
$\mathbb{M}$ to $c \mathbb{M}$ for any positive constant $c$. For example, we can take
\begin{equation}
\label{P-form}
P = \det\left(\mathbb{M}\right)^{\frac{m(p-1)}{2d}}.
\end{equation}

In the unifying form for $m \le d$, we also employ the MMPDE approach to define the moving mesh equation as a gradient system
of the mesh energy function. However, we need to pay special attention to cases $m < d$ when the mesh points are not moved out
of $S$. To this end, we denote
\begin{align}
\V{v}_i = - \frac{P_i}{\tau} \left ( \frac{\partial \mathcal{I}_h}{\partial \V{x}_i} \right )^T, \quad i = 1, ..., N_v
\label{vi-1}
\end{align}
and define
\begin{equation}
\label{MMPDE-1}
\frac{d \V{x}_i}{d t} =  \V{u}_i, \quad i = 1, ..., N_v,
\end{equation}
where $\V{u}_i$'s are related to $\V{v}_i$'s as explained in the following.

\begin{itemize}
\item[] \textbf{Case 1.} For $m = d$ (with $d = 1$, $2$, or $3$), we choose
\[
\V{u}_i = \V{v}_i, \quad i = 1, ..., N_v.
\]

\item[] \textbf{Case 2.} For $m=2$ and $d=3$, $S$ is a surface in $\mathbb{R}^3$.
        In this case, we can define
    \begin{equation}
    \label{m=d-1}
        \bm{u}_i = \bm{v}_i- \big( \bm{v}_i \cdot \bm{n}_i\big) \bm{n}_i, \quad i = 1, ..., N_v,
    \end{equation}
    where $\V{n}_i$ denotes the unit normal to $S$ at $\V{x}_i$ and can be computed using an analytical
    representation of $S$ (when available) or through a mesh/spline representation of $S$.

 \item[] \textbf{Case 3.} For $m=1$ (with $d = 2$ or $3$), $S$ is a curve in $\mathbb{R}^2$ or $\mathbb{R}^3$.
    In this case, we can define
    \begin{equation}
    \label{m=1}
        \bm{u}_i  = \big( \bm{v}_i \cdot \bm{\tau}_i\big) \bm{\tau}_i, \quad i = 1, ..., N_v,
    \end{equation}
    where $\bm{\tau}_i$ is the unit tangential vector along $S$ at $\V{x}_i$ and can be computed using an analytical
    representation of $S$ (when available) or through a mesh/spline representation of $S$.
\end{itemize}

\vspace{8pt}

To summarize, we have the unifying moving mesh equation as
\begin{equation}
\label{mmpde-unify}
\frac{d \V{x}_i}{d t} = \begin{cases}
\V{v}_i , &\text{for
~} m=d \text{ and } d = 1, 2, \text{ or } 3
\\
\bm{v}_i- \big( \bm{v}_i \cdot \bm{n}_i\big) \bm{n}_i, &\text{for
~} m=2 \text{ and } d = 3
\\
\big( \bm{v}_i \cdot \bm{\tau}_i\big) \bm{\tau}_i ,  &\text{for
~} m=1 \text{ and } d = 2 \text{ or }3
\end{cases}
\qquad i = 1, ..., N_v.
\end{equation}

\vspace{8pt}

The mesh equation \eqref{mmpde-unify} should be modified properly for the vertices on the boundary of $S$.
If $\x_i$ is a fixed boundary vertex, the corresponding equation should be replaced by
\begin{equation}
\label{mesh-fixed}
\frac{d \V x_i}{dt} = 0.
\end{equation}
If $\x_i$ is allowed to slide on the boundary, the nodal mesh velocity $\frac{d \V x_i}{dt}$ needs to be modified so that its normal component(s)
along the boundary is zero, i.e.,
\begin{equation}\label{mesh-slide}
\nabla \Phi (\V{x}_i) \cdot \frac{d \V x_i}{dt} = 0,
\end{equation}
assuming that $\Phi(\V{x})=0$ represents the bounded boundary.

\subsection{Scalar-by-matrix differentiation}
\label{sec:fA}

To compute $\V{v}_i$'s in (\ref{vi-1}) we need to compute the derivatives of $\mathcal{I}_h$ with respect to $\V{x}_i$'s.
In this subsection, we recall some definitions and properties of scalar-by-matrix differentiation \cite{HK2015}, a useful
tool for deriving the analytical expressions of those derivatives.

Let $f = f(A)$ be a scalar function of a matrix $A\in \mathbb{R}^{s\times n}$. Then the scalar-by-matrix derivative of $f$ with respect to $A$ is defined as
\begin{equation}\label{dfdA}
\frac{\partial f}{\partial A} =
\begin{bmatrix}
\frac{\partial f}{\partial A_{1,1}} & \cdots & \frac{\partial f}{\partial A_{s,1}}\\
\vdots & & \vdots\\
\frac{\partial f}{\partial A_{1,n}} & \cdots & \frac{\partial f}{\partial A_{s,n}}\\
\end{bmatrix}_{n\times s}
\quad \text{or} \quad \left(\frac{\partial f}{\partial A}\right)_{i,j} = \frac{\partial f}{\partial A_{j,i}}~.
\end{equation}
Assume that $A$ is a function of a parameter $t$, i.e., $A = A(t)$. Then,
the chain rule of differentiation with respect to $t$ is
\begin{equation}\label{dfdt}
\frac{\partial f}{\partial t}
= \sum_{i,j} \frac{\partial f}{\partial A_{j,i}}
             \frac{\partial A_{j,i}}{\partial t}
= \sum_{i,j} \left(\frac{\partial f}{\partial A}\right)_{i,j}
                   \frac{\partial A_{j,i}}{\partial t}
= \tr\left(\frac{\partial f}{\partial A}\frac{\partial A}{\partial t}\right) .
\end{equation}

\begin{lemma}
\label{dA}
(\cite{HK2015}) If $A$ is an $n\times n$ square matrix, we have
{\em
\begin{align}
&\frac{\partial \tr\left(A\right)}{\partial A} = \mathbb{I}\in \mathbb{R}^{n\times n},  \label{der-trA}
\\&\frac{\partial A^{-1}}{\partial t} = - A^{-1} \frac{\partial A}{\partial t} A^{-1} \quad \hbox{or}\quad A\frac{\partial A^{-1}}{\partial t} = - \frac{\partial A}{\partial t} A^{-1},\label{derAinv}
\\&
\frac{\partial \det\left(A\right)}{\partial t} = \det\left(A\right) ~\tr\left ( A^{-1} \frac{\partial A}{\partial t} \right ). \label{der-detA}
\end{align}}
\end{lemma}

Denote $\J= \FMF^{-1}$. Then, the mesh energy function \eqref{Ih-1} can be written in a general form as
\begin{equation}\label{Ih-2}
\mathcal{I}_h = \sum_{K\in \mathcal{T}_h} |\hat{K}| G_K
= \sum_{K\in \mathcal{T}_h} |\hat{K}| G_K\left(\J, \det\left(\J\right)\right) ,
\end{equation}
where $\J$ and $\det\left(\J\right)$ are considered as independent variables and
\begin{equation}\label{Ih-G}
G_K\left(\J, \det\left(\J\right)\right)
= \theta~ \det\left(\J\right) ^{-\frac{1}{2}}\tr\left(\J\right)^{\frac{mp}{2}}
 + (1-2\theta)~ m^{\frac{mp}{2}} \det\left(\J\right)^{\frac{p-1}{2}}.
\end{equation}

The calculation of $\partial G_K/\partial \bm{x}_i$ is through
\[
\frac{\partial G_K}{\partial \det\left(\J\right)}, \quad
\frac{\partial G_K}{\partial \J} .
\]

\begin{lemma} The derivatives of $G_K$ with
respect to the variables $\J$ and $\det\left(\J\right)$ are
{\em
\begin{align}
 \frac{\partial G_K}{\partial \J}
 &
 = \frac{ \theta mp}{2} \det\left(\J\right) ^{-\frac{1}{2}}  \tr\left(\J\right) ^{\frac{mp-2}{2}} \mathbb{I},
 \\
\frac{\partial G_K}{\partial \det\left(\J\right)}
& = -\frac{\theta}{2} \det\left(\J\right) ^{-\frac{3}{2}} \tr\left(\J\right)^{\frac{mp}{2}}
+ \frac{p-1}{2} (1-2\theta) ~m^{\frac{mp}{2}} \det\left(\J\right) ^{\frac{p-3}{2}}.\label{der-JdetJ}
\end{align}
}
\end{lemma}
\begin{proof}
The first equality can be obtained using \eqref{der-trA} while the second can be obtained by calculating
the differentiation directly.
\end{proof}

\subsection{Analytical formulas for mesh velocities}
\label{sec:ana-V}

In this subsection, we derive the analytical expressions for the mesh velocities.

Notice that $\mathcal{I}_h=\mathcal{I}_h(\mathcal{T}_h)$ is a function of the coordinates of the vertices of mesh $\mathcal{T}_h\subset \mathbb{R}^d$, i.e.,
\[
\mathcal{I}_h(\mathcal{T}_h) = \mathcal{I}_h(\bm{x}_1,...,\bm{x}_{N_v}),\quad \bm{x}_i=\left[\bm{x}_{i}^{(1)},\bm{x}_{i}^{(2)}...,\bm{x}_{i}^{(d)}\right]^{T},\quad i = 1, ..., N_v
\]
where $\V{x}_{i}^{(\ell)},~\ell = 1,...,d$ denotes the $\ell$-th component of the $i$-th vertex $\V{x}_{i} \in S\subset\mathbb{R}^d$.
We have
\begin{equation}
\label{Ix-0}
\frac{\partial \mathcal{I}_h}{\partial \bm{x}_i} = \sum_{K\in \mathcal{T}_h}   |\hat{K}| \frac{\partial G_K}{\partial \bm{x}_i}
=  \sum_{K \in \omega_i} |\hat{K}|\frac{\partial G_K}{\partial \bm{x}_{i}},\quad i = 1, ..., N_v,
\end{equation}
where $\omega_i$ is the element patch associated with $\bm{x}_i$.

Next, we give the analytical expressions for the derivatives of $G_K$ with
respect to the coordinates of all vertices of $K=\left\{\bm{x}_0^K,\bm{x}_{1}^K,...,\bm{x}_m^K\right\}$.

\begin{lemma}\label{lem:Gx}
There holds
{\em
\begin{align}
\frac{\partial G_K}{\partial [\V x_0^K,\V x_1^K, ..., \V x_m^K]}
&= \begin{bmatrix*}
-\V e^T\\
\mathbb{I}_{m\times m}
\end{bmatrix*}\frac{\partial G_K}{\partial E_K}
+\frac{1}{m+1}\sum_{j=0}^m \tr \left (\frac{\partial G_K}{\partial \mathbb{M}_K}  \mathbb{M}_{j}^{K} \right )
\begin{bmatrix*}
  \frac{\partial \phi_{j}^{K}}{\partial \x}\\
  \frac{\partial \phi_{j}^{K}}{\partial \x}\\
  \vdots \\
  \frac{\partial \phi_{j}^{K}}{\partial \x}
  \end{bmatrix*} ,
\label{der-x0m}
\end{align}
}
where $\frac{\partial \phi_{j}^{K}}{\partial \x},~j=0,1,...,m$ are given in Lemma \ref{lem-dphi},
$\V e^T = [1, ..., 1]\in \mathbb{R}^{1\times m}$, and
\begin{align}
\frac{\partial G_K}{\partial E_K}=
&-2 \left(E_K^T\mathbb{M}_K E_K\right)^{-1} \hat{E}^T\frac{\partial G_K}{\partial \J} \hat{E} \left(E_K^T\mathbb{M}_K E_K\right)^{-1}  E_K^T  \mathbb{M}_K
\nonumber\\
& -2\frac{\det(\hat{E})^{2}}{\det\left(E_K^T \mathbb{M}_K E_K\right)}~\frac{\partial G_K}{\partial \det\left(\J\right)} \left(E_K^T \mathbb{M}_K E_K\right)^{-1} E_K^T\mathbb{M}_K,  \label{dGdE-f}
\\
\frac{\partial G_K}{\partial \mathbb{M}_K} =
& - E_K \left(E_K^T\mathbb{M}_K E_K\right)^{-1} \hat{E}^T\frac{\partial G_K}{\partial \J} \hat{E} \left(E_K^T\mathbb{M}_K E_K\right)^{-1} E_K^T
\nonumber\\
& - \det(\mathbb{J}) \frac{\partial G_K}{\partial \det(\mathbb{J})} E_K\left(E_K^T \mathbb{M}_K E_K\right)^{-1} E_K^T.
\label{der-M-mesh}
\end{align}

\end{lemma}

\begin{proof}
Recall that
\begin{equation}\label{Ih-00}
G_K = G_K\left(\J, \det\left(\J\right)\right)
= \theta~ \det\left(\J\right) ^{-\frac{1}{2}}\tr\left(\J\right)^{\frac{mp}{2}}
 + (1-2\theta) ~m^{\frac{mp}{2}} \det\left(\J\right)^{\frac{p-1}{2}},
\end{equation}
where $F_K'  = E_K \hat{E}^{-1}$, $\J = \FMF^{-1}=\hat{E} \left(E_K^T\mathbb{M}_K E_K\right)^{-1} \hat{E}^T$, and
$\det\left(\J\right) = \det(\hat{E})^{2}\det\left(E_K^T \mathbb{M}_K E_K\right)^{-1}$.
When $E_K = E_K(t)$ and $\M_K = \M_K(t)$ for some parameter $t$, we have
\begin{align*}
\frac{\partial G_K}{\partial t}
&= \tr \left (\frac{\partial G_K}{\partial E_K}  \frac{\partial E_K}{\partial t} \right )
+ \tr \left (\frac{\partial G_K}{\partial \mathbb{M}_K}  \frac{\partial \mathbb{M}_K}{\partial t} \right )
\\
&=\frac{\partial G_K}{\partial t}(I)+\frac{\partial G_K}{\partial t}(II).
\end{align*}
The part $\frac{\partial G_K}{\partial t}(I)$ stands for the case of $E_K = E_K(t)$ and $\M_K$ is independent of $t$, and the part $\frac{\partial G_K}{\partial t}(II)$ stands for $\M_K = \M_K(t)$ and $E_K$ is independent of $t$.

We first consider to compute $\frac{\partial G_K}{\partial t}(I)$, i.e., $E_K = E_K(t)$ and $\M_K$ is independent of $t$.
Using the chain rule \eqref{dfdt}, we have
\begin{align}
\frac{\partial G_K}{\partial t} (I) &
= \tr \left ( \frac{\partial G_K}{\partial \J} \frac{\partial \left(\left(F_K'\right)^T\mathbb{M}_KF_K'\right)^{-1}}{\partial t} \right )
+ \frac{\partial G_K}{\partial \det\left(\J\right)} \frac{\partial \det\left(\left(F_K'\right)^T\mathbb{M}_KF_K'\right)^{-1}}{\partial t}\notag
\\
& \label{dGdt}= \tr \left ( \frac{\partial G_K}{\partial \J} \hat{E} \frac{\partial \left(E_K^T\mathbb{M}_K E_K\right)^{-1} }
{\partial t} \hat{E}^T \right )
+\det(\hat{E})^{2} \frac{\partial G_K}{\partial \det\left(\J\right)} \frac{\partial \det\left(E_K^T \mathbb{M}_K E_K\right)^{-1}}{\partial t}.
\end{align}
Moreover, using  \eqref{derAinv} we get
\begin{equation}
\begin{split}
\frac{\partial \left(E_K^T\mathbb{M}_K E_K\right)^{-1} }{\partial t}
=-\left(E_K^T\mathbb{M}_K E_K\right)^{-1} \frac{\partial \left(E_K^T\mathbb{M}_K E_K\right) }
{\partial t}  \left(E_K^T\mathbb{M}_K E_K\right)^{-1}.
\end{split}
\end{equation}
Then, we have
\begin{align*}
& \tr \left ( \frac{\partial G_K}{\partial \J} \hat{E} \frac{\partial \left(E_K^T\mathbb{M}_K E_K\right)^{-1} }
{\partial t} \hat{E}^T \right )\\
 & =  -\tr \left ( \frac{\partial G_K}{\partial \J} \hat{E}\left(E_K^T\mathbb{M}_K E_K\right)^{-1} \frac{\partial \left(E_K^T\mathbb{M}_K E_K\right) }
{\partial t}  \left(E_K^T\mathbb{M}_K E_K\right)^{-1}\hat{E}^T\right )\\
& = -\tr \left ( \frac{\partial G_K}{\partial \J} \hat{E}\left(E_K^T\mathbb{M}_K E_K\right)^{-1} \left( \frac{\partial E_K^T}{\partial t}\mathbb{M}_K E_K+ E_K^T\mathbb{M}_K\frac{\partial E_K}{\partial t}\right) \left(E_K^T\mathbb{M}_K E_K\right)^{-1}\hat{E}^T\right ).
\end{align*}
Using the trace properties $\tr(A B) = \tr(B A)$ and $\tr(A^T) = \tr(A)$, we can simplify the above equality into
\begin{align*}
& \tr \left ( \frac{\partial G_K}{\partial \J} \hat{E} \frac{\partial \left(E_K^T\mathbb{M}_K E_K\right)^{-1} }
{\partial t} \hat{E}^T \right )
\\
& = - 2\tr \left ( \left(E_K^T\mathbb{M}_K E_K\right)^{-1} \hat{E}^T  \frac{\partial G_K}{\partial \J}
\hat{E} \left(E_K^T\mathbb{M}_K E_K\right)^{-1} E_K^T \M_K \frac{\partial E_K}{\partial t}
\right ),
\end{align*}
where we have used the fact that $\partial G_K/\partial \J$ is symmetric.

Furthermore, from \eqref{derAinv} and \eqref{der-detA}  we get
\begin{align}
\frac{\partial \det\left(E_K^T \mathbb{M}_K E_K\right)^{-1}}{\partial t}
&=\det\left(E_K^T \mathbb{M}_K E_K\right)^{-1}\tr\left(\left(E_K^T \mathbb{M}_K E_K\right)\frac{\partial \left(E_K^T \mathbb{M}_K E_K\right)^{-1}}{\partial t}\right)
\notag \\
&=-\det\left(E_K^T \mathbb{M}_K E_K\right)^{-1}
\tr\left(\frac{\partial \left(E_K^T \mathbb{M}_K E_K\right)}{\partial t}\left(E_K^T \mathbb{M}_K E_K\right)^{-1}\right)
\notag \\
& = - 2 \det\left(E_K^T \mathbb{M}_K E_K\right)^{-1}
\tr\left(\left(E_K^T \mathbb{M}_K E_K\right)^{-1} E_K^T\mathbb{M}_K\frac{\partial E_K}{\partial t}\right) .
\end{align}
Thus, when $E_K = E_K(t)$ (and $\M_K$ is independent of $t$), we have
\begin{align*}
\frac{\partial G_K}{\partial t}(I)
& = - 2\tr \left ( \left(E_K^T\mathbb{M}_K E_K\right)^{-1} \hat{E}^T  \frac{\partial G_K}{\partial \J}
\hat{E} \left(E_K^T\mathbb{M}_K E_K\right)^{-1} E_K^T \M_K \frac{\partial E_K}{\partial t}
\right )
\\
& \quad -2\frac{\det(\hat{E})^{2}}{\det\left(E_K^T \mathbb{M}_K E_K\right)}~\frac{\partial G_K}{\partial \det\left(\J\right)}~\tr\left(\left(E_K^T \mathbb{M}_K E_K\right)^{-1} E_K^T\mathbb{M}_K\frac{\partial E_K}{\partial t}\right) .
\end{align*}
From the chain rule, the above equality implies that
\begin{align*}
\frac{\partial G_K}{\partial E_K}(I)
& = - 2 \left(E_K^T\mathbb{M}_K E_K\right)^{-1} \hat{E}^T  \frac{\partial G_K}{\partial \J}
\hat{E} \left(E_K^T\mathbb{M}_K E_K\right)^{-1} E_K^T \M_K
\\
& \quad -2\frac{\det(\hat{E})^{2}}{\det\left(E_K^T \mathbb{M}_K E_K\right)}~\frac{\partial G_K}{\partial \det\left(\J\right)}
\left ( E_K^T \mathbb{M}_K E_K\right)^{-1} E_K^T\mathbb{M}_K .
\end{align*}

Recalling $E_K = \left[\bm{x}_{1}^K-\bm{x}_0^K, ..., \bm{x}_m^K-\bm{x}_0^K\right] $, we have
\begin{equation}\label{dGdE}
\frac{\partial G_K}{\partial E_K}(I)
=
\frac{\partial G_K}{\partial \left[\bm{x}_{1}^K-\bm{x}_0^K, ..., \bm{x}_m^K-\bm{x}_0^K\right]}
= \begin{bmatrix*}
\frac{\partial G_K}{\partial \left(\bm{x}_{1}^K-\bm{x}_0^K\right)} \\
\frac{\partial G_K}{\partial \left(\bm{x}_{2}^K-\bm{x}_0^K\right)} \\
\vdots \\
\frac{\partial G_K}{\partial \left(\bm{x}_m^K-\bm{x}_0^K\right)}
\end{bmatrix*}
= \begin{bmatrix*}
\frac{\partial G_K}{\partial \bm{x}_{1}^K} \\
\frac{\partial G_K}{\partial \bm{x}_{2}^K} \\
\vdots \\
\frac{\partial G_K}{\partial \bm{x}_m^K}
\end{bmatrix*},
\end{equation}
where $\frac{\partial G_K}{\partial \bm{x}_{j}^K} = \frac{\partial G_K}{\partial
\left(\bm{x}_j^K-\bm{x}_0^K\right)}\frac{\partial \left(\bm{x}_j^K-\bm{x}_0^K\right)}{\partial \bm{x}_j^K}$ has been used.
Moreover, we have
\[
\frac{\partial G_K}{\partial \bm{x}_{0}^K}(I)
 = \sum_{j=1}^{m}\frac{\partial G_K}{\partial \left(\bm{x}_j^K-\bm{x}_0^K\right)}\frac{\partial \left(\bm{x}_j^K-\bm{x}_0^K\right)}{\partial \bm{x}_0^K}
= -\sum_{j=1}^{m}\frac{\partial G_K}{\partial \left(\bm{x}_j^K-\bm{x}_0^K\right)}
=- \V e^T \frac{\partial G_K}{\partial E_K},
\]
where $\V e^T = [1, ..., 1]\in \mathbb{R}^{1\times m}$.
Then, we obtain
\begin{equation}\label{dGdx0m-1}
\frac{\partial G_K}{\partial \left[\bm{x}_0^K,\bm{x}_{1}^K,...,\bm{x}_m^K\right]}(I)
= \begin{bmatrix*}
\frac{\partial G_K}{\partial \bm{x}_0^K} \\
\frac{\partial G_K}{\partial \bm{x}_{1}^K} \\
\vdots \\
\frac{\partial G_K}{\partial \bm{x}_m^K}
\end{bmatrix*}
= \begin{bmatrix*}
\frac{\partial G_K}{\partial \bm{x}_0^K} \\
\frac{\partial G_K}{\partial E_K}
\end{bmatrix*}
= \begin{bmatrix*}
-\V e^T\\
\mathbb{I}_{m\times m}
\end{bmatrix*}\frac{\partial G_K}{\partial E_K}.
\end{equation}

Next, we consider to compute $\frac{\partial G_K}{\partial t}(II)$, i.e., $\M_K = \M_K(t)$ and $E_K$ is independent of $t$.
To simplify the computation, we assume that $\mathbb{M} = \mathbb{M}(\V x)$ is a piecewise linear function
defined on the current mesh, i.e.,
\[
\mathbb{M}_{K} = \sum\limits_{j=0}^m \mathbb{M}_{K}(\bm{x}_j^K) \phi_j^K = \sum\limits_{j=0}^m \mathbb{M}_{j}^{K}\phi_j^K,
\]
where $\phi_j^K$ is the linear basis function associated with the vertex $\V x_j^K$.
Denote the barycenter of $K$ by $\V x_K = \frac{1}{m+1}\sum\limits_{j=0}^m \V x_j^K$. Then, we have
\begin{align*}
\frac{\partial G_K}{\partial t}(II) = \tr \left (\frac{\partial G_K}{\partial \mathbb{M}_K}  \frac{\partial \mathbb{M}_K}{\partial t} \right )
& = \tr \left (\frac{\partial G_K}{\partial \mathbb{M}_K}
\sum_{\ell=1}^d \frac{\partial \mathbb{M}_K}{\partial \x^{(\ell)}}
\frac{\partial \x_K^{(\ell)}}{\partial t} \right )\quad
\\
& = \tr \left (\frac{\partial G_K}{\partial \mathbb{M}_K}  \sum_{\ell=1}^d
\sum_{j=0}^m \mathbb{M}_{j}^{K}\frac{\partial \phi_{j}^{K}}{\partial \x^{(\ell)}}
\frac{\partial \x_K^{(\ell)}}{\partial t}\right )
\\
& = \frac{1}{m+1}\sum_{j=0}^m \tr \left (\frac{\partial G_K}{\partial \mathbb{M}_K}  \mathbb{M}_{j}^{K} \right )
\frac{\partial \phi_{j}^{K}}{\partial \x} \frac{\partial \x_K}{\partial t} .
\end{align*}
Taking $t$ as the entries of $[\V x_0^K,\V x_1^K, ... , \V x_m^K]$ sequentially, we obtain
\begin{align}\label{dGdx0m-2}
\frac{\partial G_K}{\partial [\V x_0^K,\V x_1^K, ... , \V x_m^K]}(II)
 &=  \frac{1}{m+1}\sum_{j=0}^m \tr \left (\frac{\partial G_K}{\partial \mathbb{M}_K}  \mathbb{M}_{j}^{K} \right )
 \begin{bmatrix*}
 \frac{\partial \phi_{j}^{K}}{\partial \x} \\
  \frac{\partial \phi_{j}^{K}}{\partial \x}\\
  \vdots \\
  \frac{\partial \phi_{j}^{K}}{\partial \x}
  \end{bmatrix*}.
\end{align}

Summarizing the above results, we have
\begin{align}
\frac{\partial G_K}{\partial [\V x_0^K,\V x_1^K, ..., \V x_m^K]}
&= \begin{bmatrix*}
-\V e^T\\
\mathbb{I}_{m\times m}
\end{bmatrix*}\frac{\partial G_K}{\partial E_K}
+\frac{1}{m+1}\sum_{j=0}^m \tr \left (\frac{\partial G_K}{\partial \mathbb{M}_K}  \mathbb{M}_{j}^{K} \right )
\begin{bmatrix*}
  \frac{\partial \phi_{j}^{K}}{\partial \x}\\
  \frac{\partial \phi_{j}^{K}}{\partial \x}\\
  \vdots \\
  \frac{\partial \phi_{j}^{K}}{\partial \x}
  \end{bmatrix*} ,
\label{der-x0m-1}
\end{align}
where
\begin{equation}\label{dGdE-f-1}
\begin{split}
\frac{\partial G_K}{\partial E_K}=
&-2 \left(E_K^T\mathbb{M}_K E_K\right)^{-1} \hat{E}^T\frac{\partial G_K}{\partial \J} \hat{E} \left(E_K^T\mathbb{M}_K E_K\right)^{-1}  E_K^T  \mathbb{M}_K
\\& -2\frac{\det(\hat{E})^{2}}{\det\left(E_K^T \mathbb{M}_K E_K\right)}~\frac{\partial G_K}{\partial \det\left(\J\right)} \left(E_K^T \mathbb{M}_K E_K\right)^{-1} E_K^T\mathbb{M}_K.
\end{split}
\end{equation}
From Lemma \ref{lem-dphi}, we have
\begin{align}
\label{der-M-mesh-1}
\frac{\partial G_K}{\partial \mathbb{M}_K}
& = - E_K \left(E_K^T\mathbb{M}_K E_K\right)^{-1} \hat{E}^T\frac{\partial G_K}{\partial \J} \hat{E} \left(E_K^T\mathbb{M}_K E_K\right)^{-1} E_K^T
\nonumber\\
& - \det(\mathbb{J}) \frac{\partial G_K}{\partial \det(\mathbb{J})} E_K\left(E_K^T \mathbb{M}_K E_K\right)^{-1} E_K^T,
\\
\label{dphi-0}
& \begin{bmatrix*}
\frac{\partial \phi_{1}^{K}}{\partial \x}\\
 \vdots \\
 \frac{\partial \phi_{m}^{K}}{\partial \x}
\end{bmatrix*}
=(E_K^TE_K)^{-1}E_K^T,\quad
\frac{\partial \phi_{0}^{K}}{\partial \x} = -\sum_{j=1}^m\frac{\partial \phi_{j}^{K}}{\partial \x}.
\end{align}

\end{proof}

\section{Mesh nonsingularity}
\label{sec:mesh-non}
In this section, we study the nonsingularity of the mesh trajectory $\mathcal{T}_h(t)$ at time $t$ and the existence of limiting meshes as $t \to \infty$ for the unifying moving mesh equation \eqref{mmpde-unify}.
For notational simplicity, we assume in this section that $\hat{K}$ is taken to satisfy
$|\hat{K}| = \frac{1}{N}$ instead of being unitary.
This is because typically we expect $|K| = \mathcal{O}(\frac{1}{N})$ and this
assumption will likely lead to $F_K' = \mathcal{O}(1)$ and
$\mathcal{I}_h(\mathcal{T}_h(0))=\mathcal{O}(1)$.
Note that this assumption is for notational simplicity and does not affect the implementation and theoretical properties of
the moving mesh method.

Recall that the mesh energy function in a general form is
\begin{equation}\label{Ih-n}
\mathcal{I}_h = \sum_{K\in \mathcal{T}_h} |\hat{K}| G_K
= \sum_{K\in \mathcal{T}_h} |\hat{K}| G_K\left(\J, \det\left(\J\right)\right) ,
\end{equation}
where $\J= \FMF^{-1}$ and $\det\left(\J\right)$ are considered as independent variables,
\begin{equation}
G_K\left(\J, \det\left(\J\right)\right)
= \theta~ \det\left(\J\right) ^{-\frac{1}{2}}\tr\left(\J\right)^{\frac{mp}{2}}
 + (1-2\theta) m^{\frac{mp}{2}} \det\left(\J\right)^{\frac{p-1}{2}},
\label{GK-1}
\end{equation}
and $ p > 1$ and $\theta \in (0,1]$ are dimensionless parameters.
It is worth pointing out that while we focus on the function (\ref{GK-1}) in this work,
we carry out the analysis in this section for a more general function $G_K$ that is assumed to satisfy
the coercivity condition
\begin{equation}
\label{coer}
G_K \left(\mathbb{J}, \det\left(\mathbb{J}\right)\right) \ge \det\left(\J\right) ^{-\frac{1}{2}} \left[\alpha~ \tr(\J)^q
- \beta\right], \quad \quad \forall \det\left(\mathbb{J}\right ) \ge 0,
\end{equation}
where $q> \frac{m}{2}$, $\alpha>0$, and $\beta\ge 0$ are constants.
Notice that the function in (\ref{GK-1}) satisfies (\ref{coer}) with $\alpha = \theta$, $\beta = 0$, and $q = \frac{m p}{2}$
when $p > 1$ and $\theta \in (0,0.5]$.

From \eqref{mmpde-unify}, we can rewrite the unifying moving mesh equation as
\begin{equation}\label{mmpde-unify-n}
\frac{d \V{x}_i}{d t} = \begin{cases}
- \frac{P_i}{\tau} \left ( \frac{\partial \mathcal{I}_h}{\partial \V{x}_i} \right )^T , &\text{for
~} m=d \text{ and } d = 1, 2, \text{ or }3,
\\
- \frac{P_i}{\tau} \left( \left ( \frac{\partial \mathcal{I}_h}{\partial \V{x}_i} \right )^T- \left( \left ( \frac{\partial \mathcal{I}_h}{\partial \V{x}_i} \right )^T \cdot \bm{n}_i\right) \bm{n}_i\right), &\text{for
~} m=2 \text{ and } d = 3,
\\
- \frac{P_i}{\tau}\left( \left ( \frac{\partial \mathcal{I}_h}{\partial \V{x}_i} \right )^T \cdot \bm{\tau}_i\right) \bm{\tau}_i ,  &\text{for
~} m=1 \text{ and } d = 2 \text{ or } 3,
\end{cases}
\end{equation}
where $\V{n}_i$ and $\bm{\tau}_i$ stand for the unit normal and tangent vectors to $S$ at $\V{x}_i$, respectively.

\begin{lemma}
\label{lem:Id-n}
Along any mesh trajectory governed by the equation \eqref{mmpde-unify-n}, the energy function decreases as $t$ increases.
As a consequence, $\mathcal{I}_h(\mathcal{T}_h(t))\leq \mathcal{I}_h(\mathcal{T}_h(0))$ for any time $t>0$.
\end{lemma}
\begin{proof}
From \eqref{mmpde-unify-n}, for the case of $m=d$ and $ d = 1, 2$, or 3,
we have
\begin{align*}
\frac{d\mathcal{I}_h}{dt}
& = \sum_{i} \frac{\partial \mathcal{I}_h}{\partial  \V x_i} \frac{d \V x_i}{dt}
 = -\sum_{i} \frac{P_i}{\tau}\frac{\partial \mathcal{I}_h}{\partial  \V x_i}
\left ( \frac{\partial \mathcal{I}_h}{\partial \V{x}_i} \right )^T
\\& = -\sum_{i} \frac{P_i}{\tau} \left \| \frac{\partial \mathcal{I}_h}{\partial \V{x}_i} \right \|^2  \le 0.
\end{align*}
For the case of $m=2$ and $ d = 3$, we have
\begin{align*}
\frac{d\mathcal{I}_h}{dt} & = \sum_{i} \frac{\partial \mathcal{I}_h}{\partial  \V x_i} \frac{d \V x_i}{dt}
= -\sum_{i} \frac{P_i}{\tau}\frac{\partial \mathcal{I}_h}{\partial  \V x_i} \left [ \left ( \frac{\partial \mathcal{I}_h}{\partial \V{x}_i} \right )^T - \left (\left (\frac{\partial \mathcal{ I}_h}{\partial \V{x}_i}\right )^T \cdot \V{n}_i \right ) \V{n}_i \right ]\\
& = -\sum_{i} \frac{P_i}{\tau} \left [ \left \| \frac{\partial \mathcal{I}_h}{\partial \V{x}_i} \right \|^2 - \left (\left (\frac{\partial \mathcal{I}_h}{\partial \V{x}_i}\right )^T \cdot \V{n}_i \right )^2 \right ]
 \le 0.
\end{align*}
For the case of $m=1 $  and $ d = 2$ or 3, we have
\begin{align*}
\frac{d\mathcal{I}_h}{dt}
& = \sum_{i} \frac{\partial \mathcal{I}_h}{\partial  \V x_i} \frac{d \V x_i}{dt}
 = -\sum_{i} \frac{P_i}{\tau}\frac{\partial \mathcal{I}_h}{\partial  \V x_i} \left [
\left( \left ( \frac{\partial \mathcal{I}_h}{\partial \V{x}_i} \right )^T \cdot \bm{\tau}_i\right) \bm{\tau}_i
\right ]
\\& = -\sum_{i} \frac{P_i}{\tau} \left( \left ( \frac{\partial \mathcal{I}_h}{\partial \V{x}_i} \right )^T \cdot \bm{\tau}_i\right)^2
 \le 0.
\end{align*}
Thus, $\mathcal{I}_h\left(\mathcal{T}_h(t)\right)$ is a decreasing function of $t$.
\end{proof}

\begin{thm}
\label{thm:K-ns}
We assume that the mesh energy function satisfies the coercivity condition (\ref{coer}) and
the metric tensor $\mathbb{M}(\bm{x})$ is sufficiently smooth and bounded below and above as
\begin{equation}
\label{Mbound-0}
\underline{\rho} \; \mathbb{I}\le\mathbb{M}(\bm{x}) \le \overline{\rho} \; \mathbb{I},
\quad \forall \bm{x}\in S\subset \mathbb{R}^d,
\end{equation}
where $\underline{\rho}$ and $\overline{\rho}$ are two positive constants.
Then, we have
\begin{enumerate}
  \item[(i)] The elements of the mesh trajectory $\mathcal{T}_h(t)$
  of the unifying moving mesh equation \eqref{mmpde-unify-n} have positive size for all time $t > 0$ if they have positive size initially.
  \item[(ii)] The minimum altitudes of the elements of $\mathcal{T}_h(t)$ in the metric $\mathbb{M}$
  and their areas in the Euclidean metric are bounded below by
\begin{align}
\label{amin}
a_{K,\mathbb{M}} &\ge C_1 ~\left[\mathcal{I}_h(\mathcal{T}_h(0))
+ \beta \left(\overline{\rho}\right)^{\frac{d}{2}} |S|\right]^{-\frac{1}{2q-m}} ~N^{-\frac{2q}{m(2q-m)}},
\\[2pt]
\label{kmin}
|K| & \ge C_2~\left[\mathcal{I}_h(\mathcal{T}_h(0)) + \beta \left(\overline{\rho}\right)^{\frac{d}{2}} |S|\right]^{-\frac{m}{2q-m}} ~N^{-\frac{2q}{2q-m}}~\left(\overline{\rho}\right)^{-\frac{d}{2}},
\end{align}
where
\begin{equation}\label{coeff}
C_1 = \left(\frac{\alpha~ \lambda^{2q}}{m^{\frac{m}{2}}m!}\right)^{\frac{m}{2q-m}}, \qquad C_2 = \dfrac{\left(C_1\right)^{m}}{m^{\frac{m}{2}}m!},
\qquad \lambda = \frac{\sqrt{m+1}~(m!)^{\frac{1}{m}}}{\sqrt{m}~(m+1)^{\frac{1}{2m}}}.
\end{equation}
\end{enumerate}
\end{thm}

\begin{proof}
From Lemma \ref{lem:Gx} it is not difficult to see that
the magnitude of the mesh velocities is bounded from above when $|K|$ is bounded from below by a positive constant.
As a consequence, the mesh vertices will move continuously with time, and $|K|$ will not jump over the bound to become negative.
Hence, $|K|$ will stay positive if it is so initially. Thus, the key of this proof is to show that
(\ref{kmin}) holds for any $t > 0$.

Using \eqref{Mbound-0} and Lemma \ref{lem:volK} and Lemma \ref{lem:volK-M}, we have
\begin{align}
|K|_{\mathbb{M}} &
=|\hat{K}|\det\left(\J\right) ^{-\frac{1}{2}}= |\hat{K}|\det\FMF^{\frac12}
\notag
\\
&\le \left(\overline{\rho}\right)^{\frac{d}{2}}|\hat{K}|
\det\left(\left(F_K'\right)^TF_K'\right)^{\frac12} = \left(\overline{\rho}\right)^{\frac{d}{2}}|K|.
\label{Km}
\end{align}
Summing the above inequality overall $K$, we get
\begin{equation}
\label{sumK-M}
\sum_{K\in \mathcal{T}_h} |K|_{\mathbb{M}}
\le \left(\overline{\rho}\right)^{\frac{d}{2}}|S|.
\end{equation}
Moreover, from Lemma \ref{lem:volKM-aK} we have
\begin{align}\label{aK-M-n}
|K|_{\mathbb{M}}=|\hat{K}|\det\left(\J\right) ^{-\frac{1}{2}}\geq \frac{a_{K,\mathbb{M}}^{m}}{m^{\frac{m}{2}}m!}.
 \end{align}
This means that
\begin{align*} \frac{a_{K,\mathbb{M}}^{m}}{m^{\frac{m}{2}}m!}
\leq |K|_{\mathbb{M}}
\leq\left(\overline{\rho}\right)^{\frac{d}{2}}|K|
 \end{align*}
and
\begin{align}\label{Keq-n}
|K|\geq \left(\overline{\rho}\right)^{-\frac{d}{2}}
\frac{a_{K,\mathbb{M}}^{m}}{m^{\frac{m}{2}}m!}.
\end{align}

From the above results and the coercivity condition (\ref{coer}) and Lemma \ref{lem:eq-tr-norm}, we get
\begin{align*}
\mathcal{I}_h\left(\mathcal{T}_h(t)\right)
&= \sum_{K\in \mathcal{T}_h} |\hat{K}| G_K\left(\J, \det\left(\J\right)\right)
\\& \ge \sum_{K\in \mathcal{T}_h} |\hat{K}| \det\left(\J\right) ^{-\frac{1}{2}} \left[\alpha~ \tr(\J)^q
- \beta\right]\\
& \ge \alpha\sum_{K\in \mathcal{T}_h} |\hat{K}| \det\left(\J\right) ^{-\frac{1}{2}}  \left\|\J\right\|^q
- \beta \left(\overline{\rho}\right)^{\frac{d}{2}} |S|\\
& \ge \alpha \sum_{K\in \mathcal{T}_h}  |\hat{K}|\det\left(\J\right) ^{-\frac{1}{2}} \frac{\hat{a}^{2q}}{a_{K,\mathbb{M}}^{2q}}
 - \beta \left(\overline{\rho}\right)^{\frac{d}{2}} |S| .
\end{align*}
Using \eqref{aK-M-n}, we have
\begin{equation}
\mathcal{I}_h\left(\mathcal{T}_h(t)\right) +\beta \left(\overline{\rho}\right)^{\frac{d}{2}} |S|
\ge
\alpha \sum_{K\in \mathcal{T}_h}  \frac{a_{K,\mathbb{M}}^{m}}{m^{\frac{m}{2}}m!}~ \frac{\hat{a}^{2q}}{a_{K,\mathbb{M}}^{2q}}
=\frac{\alpha~\hat{a}^{2q}}{m^{\frac{m}{2}}m!} \sum_{K\in \mathcal{T}_h} \dfrac{1}{a_{K,\mathbb{M}}^{2q-m}}~.
\end{equation}
Therefore, from Lemma \ref{lem:Id-n} we have
\begin{equation}\label{akm2}
a_{K,\mathbb{M}}^{2q-m} \ge \frac{\alpha ~\hat{a}^{2 q}}{m^{\frac{m}{2}}m!}\left(\mathcal{I}_h\left(\mathcal{T}_h(0)\right) +\beta \left(\overline{\rho}\right)^{\frac{d}{2}} |S|\right)^{-1}.
\end{equation}
Moreover, from the assumption that $\hat{K}$ is equilateral and $|\hat{K}| = \frac{1}{N} $ it follows that
\begin{equation}
\label{ahat}
\hat{a} = \frac{\sqrt{m+1}~(m!)^{\frac{1}{m}}}{\sqrt{m}~(m+1)^{\frac{1}{2m}}} ~N^{-\frac{1}{m}}.
\end{equation}
Let $\lambda = \frac{\sqrt{m+1}~(m!)^{\frac{1}{m}}}{\sqrt{m}~(m+1)^{\frac{1}{2m}}} $.
Substituting \eqref{ahat} into \eqref{akm2}, we get (\ref{amin}).
Moreover, substituting \eqref{akm2} into \eqref{Keq-n} gives (\ref{kmin}).

\end{proof}

\vspace{8pt}

\begin{rem}
Recall that the condition (\ref{coer}) is satisfied by the meshing function (\ref{GK-1})
for $\theta \in (0,0.5]$ and $p>1$. In this case,
the role of the parameter $p$ can be explained from (\ref{amin}).
Indeed, the inequality (\ref{amin}) becomes
\begin{equation*}
a_{K,\mathbb{M}} \ge C_1 ~\left[\mathcal{I}_h(\mathcal{T}_h(0)) + \beta \left(\overline{\rho}\right)^{\frac{d}{2}} |S|\right]^{-\frac{1}{m(p-1)}} ~N^{-\frac{p}{m(p-1)}} \to C_1~N^{-\frac{1}{m}}, \quad p\to \infty.
\end{equation*}
Since $N^{-\frac{1}{m}}$ represents the average diameter of the elements, the above inequality implies
that the mesh becomes more uniform as $p$ gets larger.
\qed
\end{rem}

\begin{thm}
\label{thm:limiting}
Under the assumptions of Theorem \ref{thm:K-ns}, for any nonsingular initial mesh, the mesh trajectory $\{\mathcal{T}_h(t), t>0\}$ of the unifying moving mesh equation \eqref{mmpde-unify-n} has the following properties.
\begin{enumerate}
	\item[(i)] $\mathcal{I}_h(\mathcal{T}_h(t))$ has a limit as $t\to \infty$, i.e.,
	\begin{equation}
\lim_{t\to \infty} \mathcal{I}_h(\mathcal{T}_h(t)) = L,
\end{equation}
where $L$ is a bounded constant.
	\item[(ii)] The mesh trajectory has limiting meshes, all of which are nonsingular and satisfy (\ref{amin}) and (\ref{kmin}).
	\item[(iii)] The limiting meshes are critical points of $\mathcal{I}_h$, i.e., they satisfy
	\begin{equation}\label{critical}
\frac{\partial \mathcal{I}_h}{\partial \V x_i} = 0, \quad i = 1, \dots, N_v.
\end{equation}
\end{enumerate}
\end{thm}
\begin{proof}
The proof is similar to that of \cite[Theorem 4.9]{HK2018}.
\end{proof}


\section{Numerical examples}
\label{sec:numerical}

In this section, we present numerical results for curves ($m=1$) and surfaces ($m=2$) in $\mathbb{R}^2$ and $\mathbb{R}^3$
to demonstrate the ability of the unifying moving mesh method described in the previous sections to move and concentrate
the mesh points.
In our computation, initial meshes are obtained through the analytical expressions of the corresponding curves/surfaces
with random perturbation.

The metric tensor $\mathbb{M}$ is assumed to be symmetric and uniformly positive definite on geometry $S$. It is used to control the size, shape, and orientation of the elements of the mesh to be generated.
The metric tensor is typically defined based on error estimates or indicators \cite{H2005-M,HR2011-book}, or other physical and geometric considerations \cite{AvaryH2020}.
In this work, we consider two metric tensors.
The first is the identity metric tensor $\mathbb{M}=\mathbb{I}$, with which the mesh moves to become more uniform in the Euclidean metric.
The second one is a curvature-based metric tensor defined as $\mathbb{M} = \bar{k}\; \mathbb{I}$,
where $\bar{k}$ is the mean curvature of the underlying curve/surface \cite{Curvature}.
With this metric tensor, we expect the final mesh to concentrate more vertices in regions with larger mean curvatures and fewer vertices in areas with smaller mean curvatures.

In our computation, the mean curvature and the normal vector (for surface) or the tangential vector (for curve)
are approximated numerically regardless of the availability of an analytical parametric representation of the underlying curve/surface.
(The reader is referred to, e.g., \cite{Ksurface,Kcurve} for the numerical approximation of the mean curvature.)
It is worth pointing out that the unifying moving mesh method does not require the availability of an
analytical parametric representation of the underlying curve/surface. Numerical approximation of
the needed information on the normal/tangential vector for $S$
can be obtained using the initial or current mesh representation of $S$.

\begin{example}\label{ex:ellipse}
(2D Circle and Ellipse)
\end{example}
In this first example, we test the unifying moving mesh method with the Euclidean and curvature-based metric tensors
for circles and ellipses in $\mathbb{R}^2$ that have the parametric representation as
\begin{equation}\label{circle-m1}
x(s) = r\cos(s), \quad
y(s) = \sin(s),\quad s\in [0,2\pi],\quad r\geq 1,
\end{equation}
with $r=1$ (unit circle) and $r=6$ (ellipse).
Two randomly perturbed nonuniform initial (curve) meshes are shown in Figs.~\ref{Fig:Circle}(a) and \ref{Fig:Ellipse}(a).
It is clear that the initial meshes are distributed unevenly along the curve.

The final meshes for the Euclidean metric tensor $\mathbb{M} = \mathbb{I}$ and
$\mathbb{M} = \bar{k}\,\mathbb{I}$ are shown in Fig.~\ref{Fig:Circle}.
We can see that the final meshes in Fig.~\ref{Fig:Circle} for the unit circle are distributed evenly along the curve for
both metric tensors. These results demonstrate the effective control of mesh concentration by the unifying moving mesh method
since with $\mathbb{M} = \mathbb{I}$, the method intends to make the final mesh more uniformly distributed.
The meshes in Fig.~\ref{Fig:Circle}(b) and (c) are almost identical since the circle has constant curvature
and the Euclidean and curvature-based metric tensors are the same for this case.

The situation is different for the ellipse case. The final meshes are shown in Fig.~\ref{Fig:Ellipse} for this case.
The mesh (in Fig.~\ref{Fig:Ellipse}(b)) for the Euclidean metric is evenly distributed along the ellipse while
the mesh (in Fig.~\ref{Fig:Ellipse}(c)) for the curvature-based metric tensor $\mathbb{M} = \bar{k}\,\mathbb{I}$ has
higher concentration near the ends at $(\pm 6, 0)$ where the curvature is larger.
Again, these results confirm the ability of the unifying moving mesh method to control effectively
the mesh concentration through the metric tensor.

 \begin{figure}[H]
 \centering
 \subfigure[Initial mesh]{
 \includegraphics[width=0.31\textwidth, trim=20 0 20 20,clip]{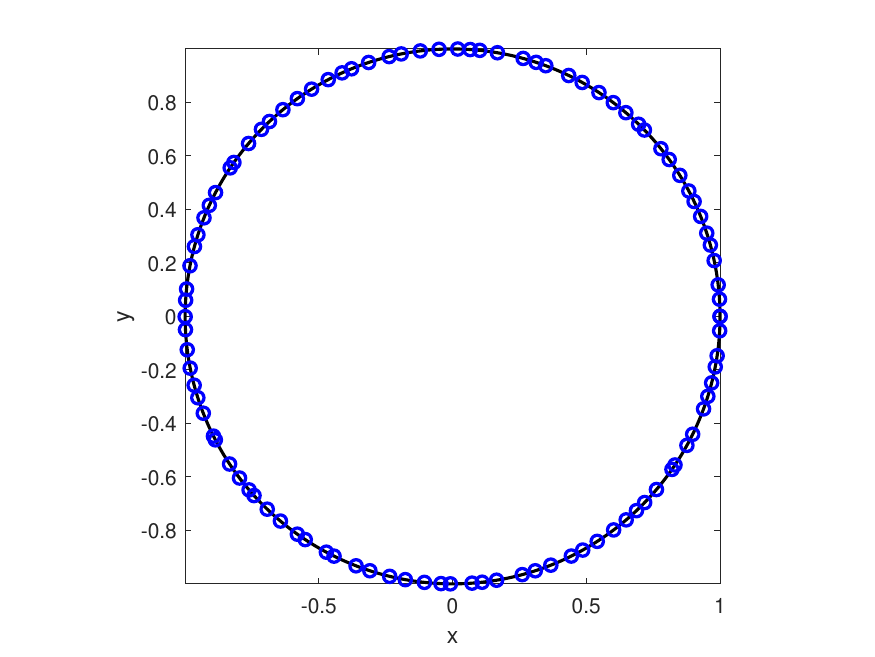}}
 \subfigure[Final mesh: $\mathbb{M}=\mathbb{I}$]{
 \includegraphics[width=0.31\textwidth, trim=20 0 20 20,clip]{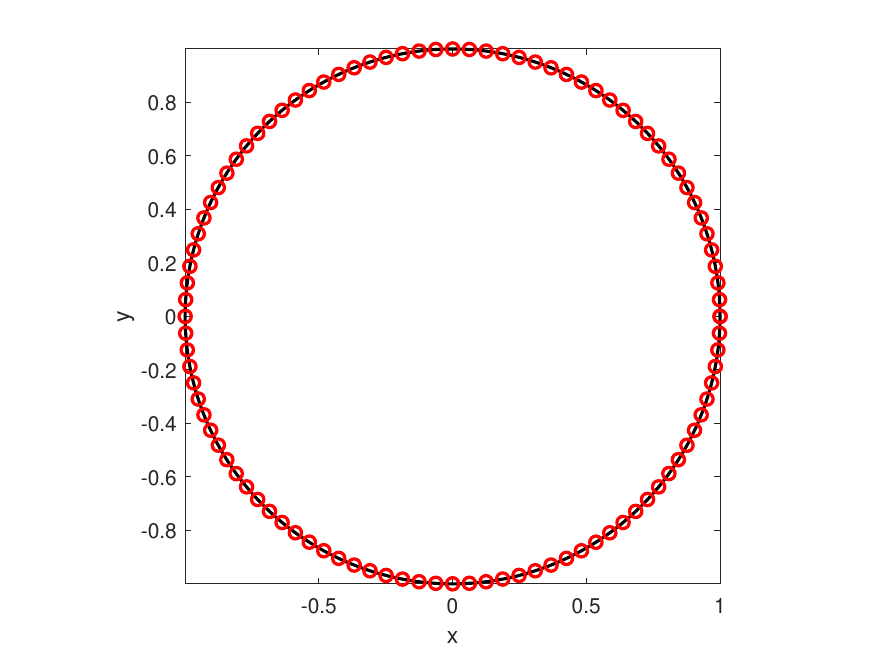}}
 \subfigure[Final mesh: $\mathbb{M} = \bar{k}\,\mathbb{I}$]{
 \includegraphics[width=0.31\textwidth, trim=20 0 20 20,clip]{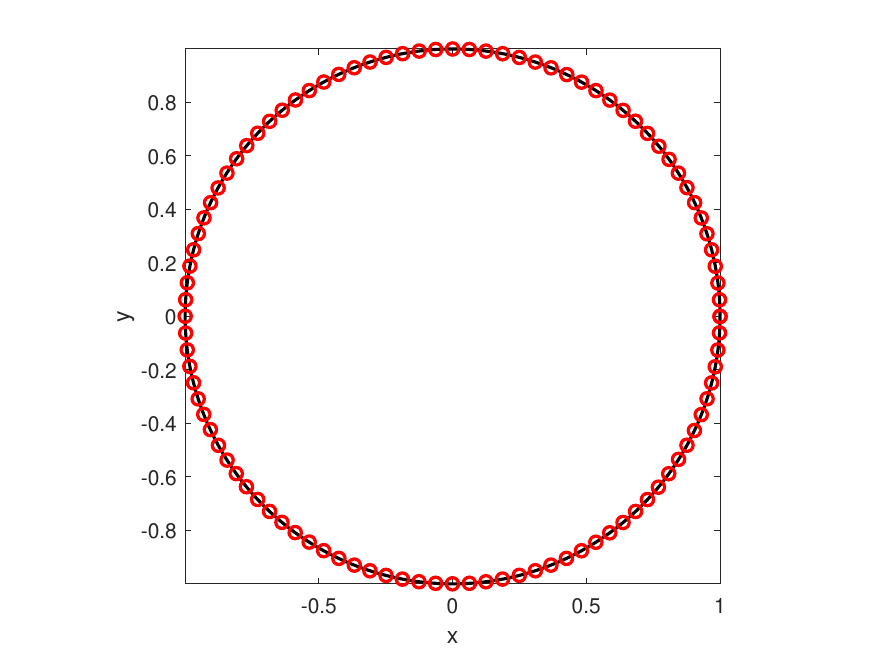}}
 \caption{{\bf Example} \ref{ex:ellipse}.
Meshes ($N=100$) of a unit circle are obtained by the unifying MMPDE method.
 }
 \label{Fig:Circle}
 \end{figure}

\begin{figure}[H]
 \centering
 \subfigure[Initial mesh]{
 \includegraphics[width=0.31\textwidth, trim=20 100 20 120,clip]{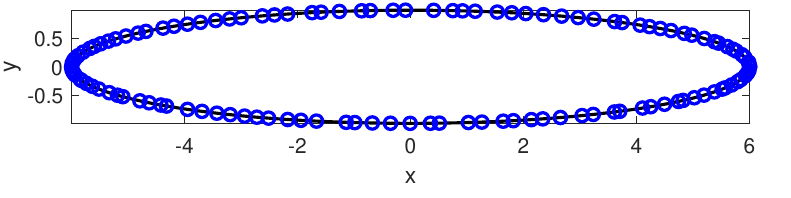}}
 \subfigure[Final mesh: $\mathbb{M}=\mathbb{I}$]{
 \includegraphics[width=0.31\textwidth, trim=20 100 20 120,clip]{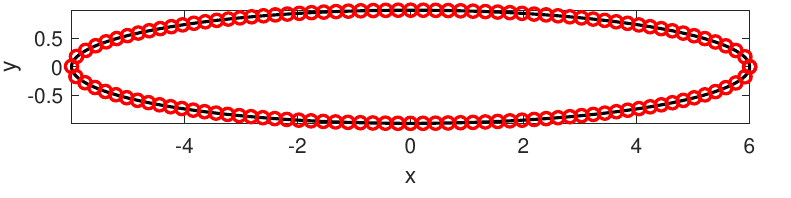}}
 \subfigure[Final mesh: $\mathbb{M} = \bar{k}\,\mathbb{I}$]{
 \includegraphics[width=0.31\textwidth, trim=20 100 20 120,clip]{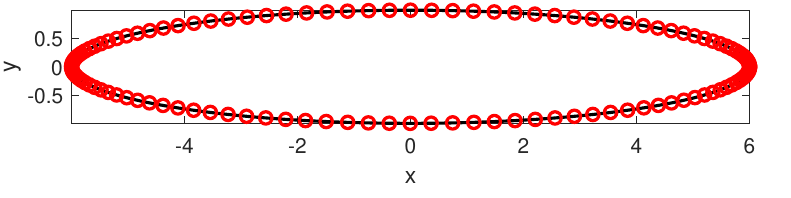}}
 \caption{{\bf Example} \ref{ex:ellipse}.
 Meshes ($N=120$) of an ellipse are obtained by the unifying MMPDE method.
 }
 \label{Fig:Ellipse}
 \end{figure}


\begin{example}\label{ex:lemniscate}
 (2D Lemniscate curve)
\end{example}
In this example, we consider the 2D lemniscate curve in $\mathbb{R}^2$ that has the parametric representation as
\begin{equation}
x(s) = \frac{2\cos(s)}{ 1+\sin^2(s)}, \quad
y(s) = \frac{\sin( s) \cos(s)}{ 1+\sin^2(s)},\quad s\in [0,2\pi].
\end{equation}
Randomly perturbed nonuniform initial mesh is shown in Fig.~\ref{Fig:lemniscate}(a)).

The final meshes for both the Euclidean metric $\mathbb{M} = \mathbb{I}$ and curvature-based metric $\mathbb{M} = \bar{k}\,\mathbb{I}$
are shown in Figs.~\ref{Fig:lemniscate}(b) and (c), respectively.
The final mesh with $\mathbb{M} = \mathbb{I}$ is evenly distributed along the curve
while the final mesh with $\mathbb{M} = \bar{k}\,\mathbb{I}$ has a higher concentration of vertices on the segments near
$(\pm 2, 0)$ where the curvature is larger than on the rest of the curve. This example also shows that the unifying moving mesh method
does not encounter any difficulty in moving the vertices along the curve even when the curve self-intersects.

 \begin{figure}[H]
 \centering
 \subfigure[Initial mesh]{
 \includegraphics[width=0.31\textwidth, trim=20 70 20 90,clip]{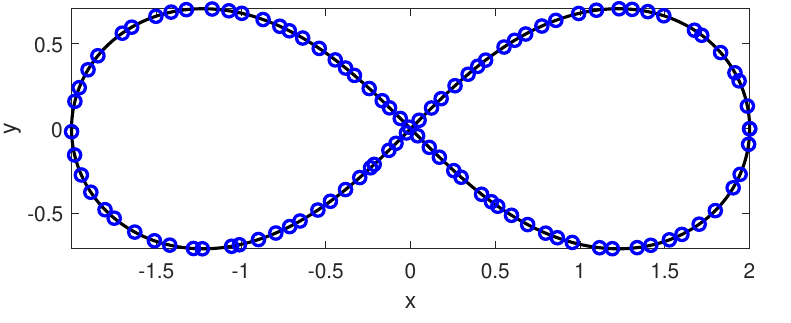}}
  \subfigure[Final mesh: $\mathbb{M}=\mathbb{I}$]{
 \includegraphics[width=0.31\textwidth, trim=20 70 20 90,clip]{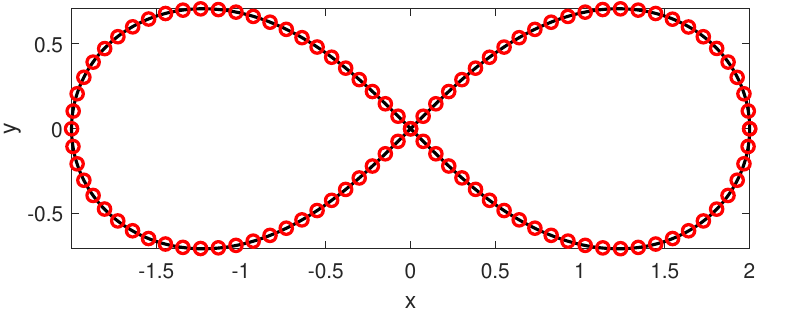}}
 \subfigure[Final mesh: $\mathbb{M} = \bar{k}\,\mathbb{I}$]{
 \includegraphics[width=0.31\textwidth, trim=20 70 20 90,clip]{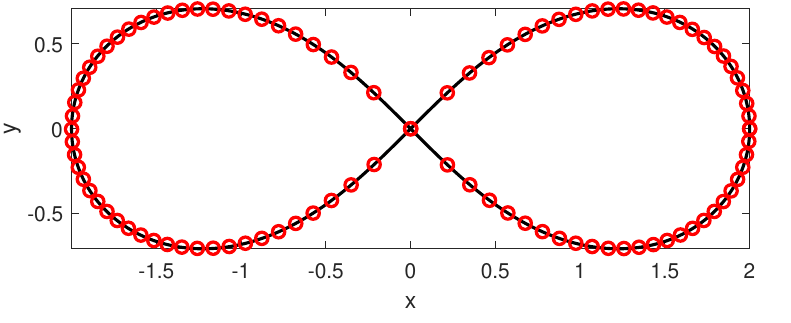}}
 \caption{{\bf Example} \ref{ex:lemniscate}.
 Meshes ($N=100$) of the lemniscate curve are obtained by the unifying  MMPDE method. }
 \label{Fig:lemniscate}
 \end{figure}

\begin{example}\label{ex:cardioid}
 (2D Cardioid curve)
\end{example}
In this example, we consider the 2D cardioid curve having the parametric representation as
\begin{equation}
x(s) = 2\cos(s)(1-\cos(s)), \quad
y(s) =  2\sin(s)(1-\cos(s)),\quad s\in [0,2\pi].
\end{equation}
A randomly perturbed nonuniform initial mesh ($N=70$) is shown in Fig.~\ref{Fig:cardioid}.
In the figure, final meshes obtained with the Euclidean metric $\mathbb{M} = \mathbb{I}$ and curvature-based metric $\mathbb{M} = \bar{k}\,\mathbb{I}$ are also shown.
We can see that the final mesh with $\mathbb{M} = \mathbb{I}$ is evenly distributed along the curve
while the mesh with the curvature-based metric $\mathbb{M} = \bar{k}\,\mathbb{I}$ has higher concentration near
(0,0) where the curvature is larger than the rest of the curve.

 \begin{figure}[H]
 \centering
 \subfigure[Initial mesh]{
 \includegraphics[width=0.31\textwidth, trim=30 0 50 20,clip]{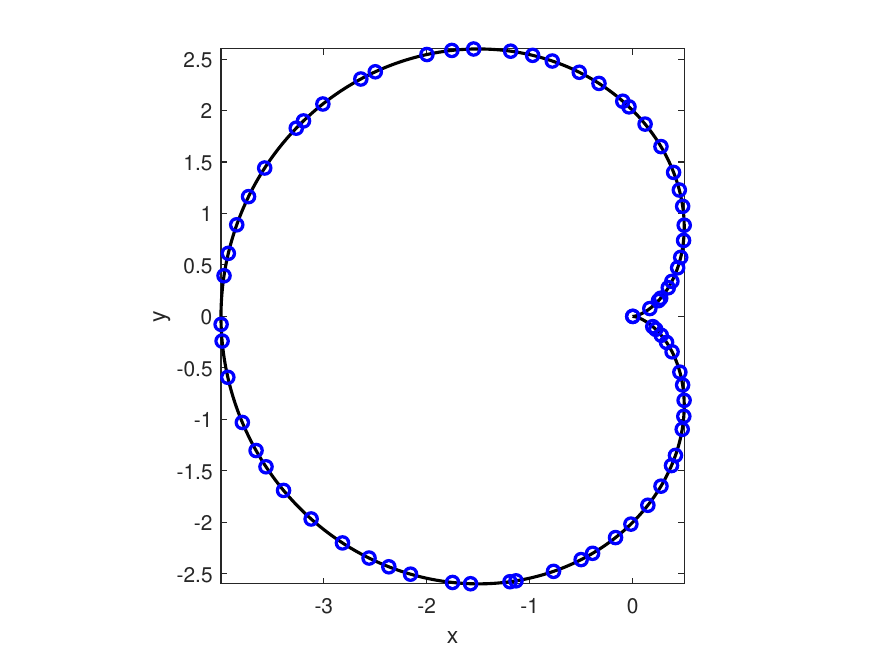}}
  \subfigure[Final mesh: $\mathbb{M}=\mathbb{I}$]{
 \includegraphics[width=0.31\textwidth, trim=30 0 50 20,clip]{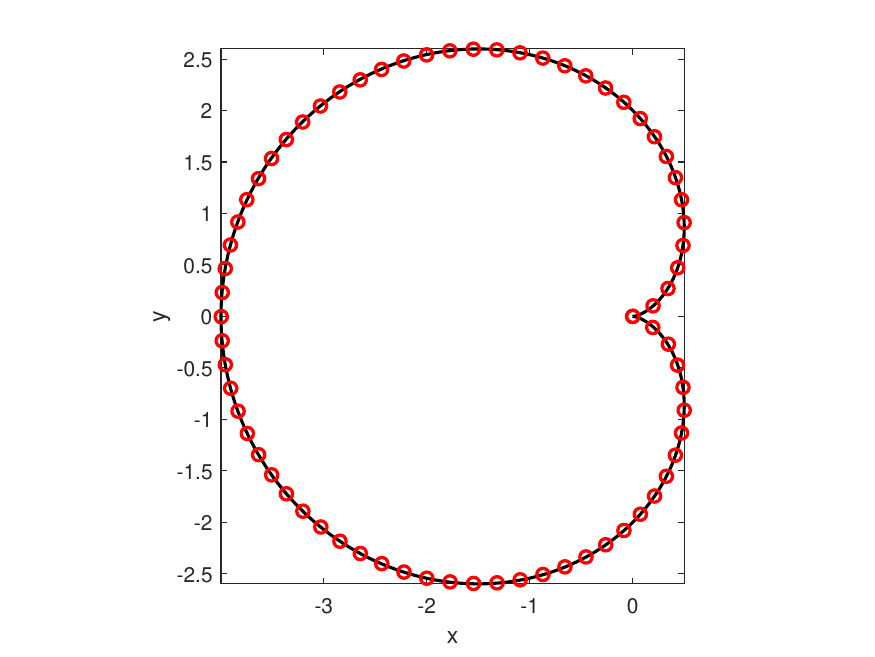}}
 \subfigure[Final mesh: $\mathbb{M} = \bar{k}\,\mathbb{I}$]{
 \includegraphics[width=0.31\textwidth, trim=30 0 50 20,clip]{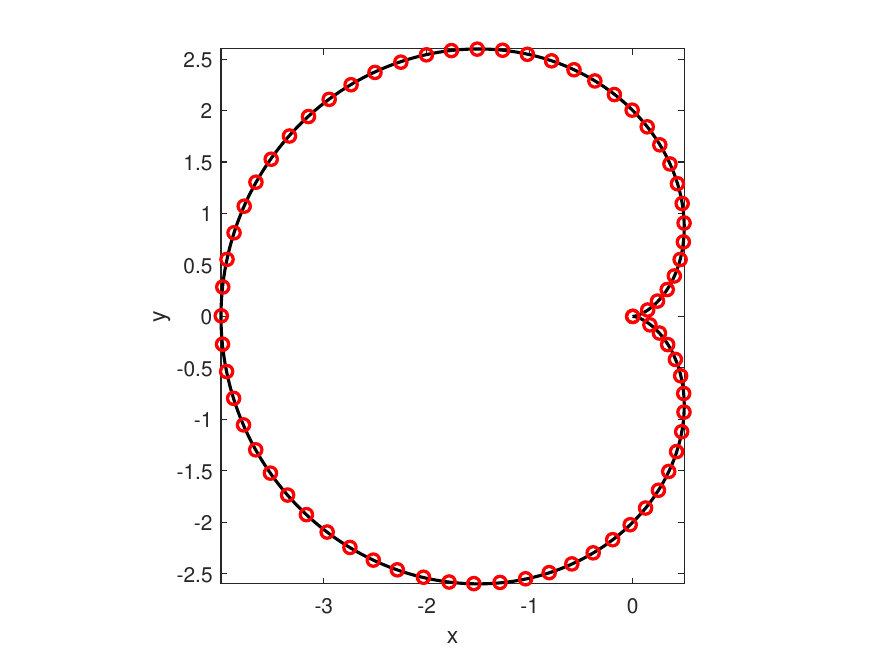}}
 \caption{{\bf Example} \ref{ex:cardioid}.
 Meshes ($N=70$) of the cardioid curve are obtained by the unifying MMPDE method.
 }
 \label{Fig:cardioid}
 \end{figure}

\begin{example}\label{ex:rose}
(2D Rose curves cluster)
\end{example}
As the final example for 2D curves, we consider a set of 2D rose curves.
They have the parametric representation as
\begin{equation}\label{rose-m1}
x(s) = \cos(r s) \cos(s), \quad
y(s) = \cos(r s) \sin(s),\quad s\in [0,c\pi],
\end{equation}
where $c$ is a constant.
We consider the following four settings.
\begin{itemize}
    \item[]\textbf{Case 1}: $r = \frac{1}{6}$ and $ c = 3$. The curve is a spiral. The initial mesh ($N=100$)
    is given in Fig.~\ref{Fig:rose6}(a).
    \item[]\textbf{Case 2}: $r = \frac{2}{3}$ and $ c = 6$. The curve looks like four intersecting cardioids.
    The initial mesh ($N=200$) is given in Fig.~\ref{Fig:rose3}(a).
    \item[]\textbf{Case 3}: $r = \frac{1}{4}$ and $c = 8$. The curve consists of two limacons inside a unit circle.
    The initial mesh ($N=200$) is given in Fig.~\ref{Fig:rose2}(a).
    \item[]\textbf{Case 4}: $r = 4$ and $c = 2$. The curve is an eight-petal rose.
    The initial mesh ($N=240$) is given in Fig.~\ref{Fig:rose1}(a).
\end{itemize}

The final meshes for both the Euclidean metric $\mathbb{M} = \mathbb{I}$ and curvature-based metric $\mathbb{M} = \bar{k}\,\mathbb{I}$
are shown in Figs.~\ref{Fig:rose6}, ~\ref{Fig:rose3}, ~\ref{Fig:rose2}, and ~\ref{Fig:rose1}, respectively.
We can see that the unifying moving mesh method works well for these complex (and self-intersecting) curves.
As for previous examples, the final meshes with $\mathbb{M}_K = \mathbb{I}$ are evenly distributed along the curves
while those meshes with the curvature-based metric $\mathbb{M} = \bar{k}\,\mathbb{I}$ have higher vertex concentration
in regions with larger curvature. These are consistent with the expectation of those metric tensors.

 \begin{figure}[H]
 \centering
 \subfigure[Initial mesh]{
 \includegraphics[width=0.31\textwidth, trim=20 0 20 20,clip]{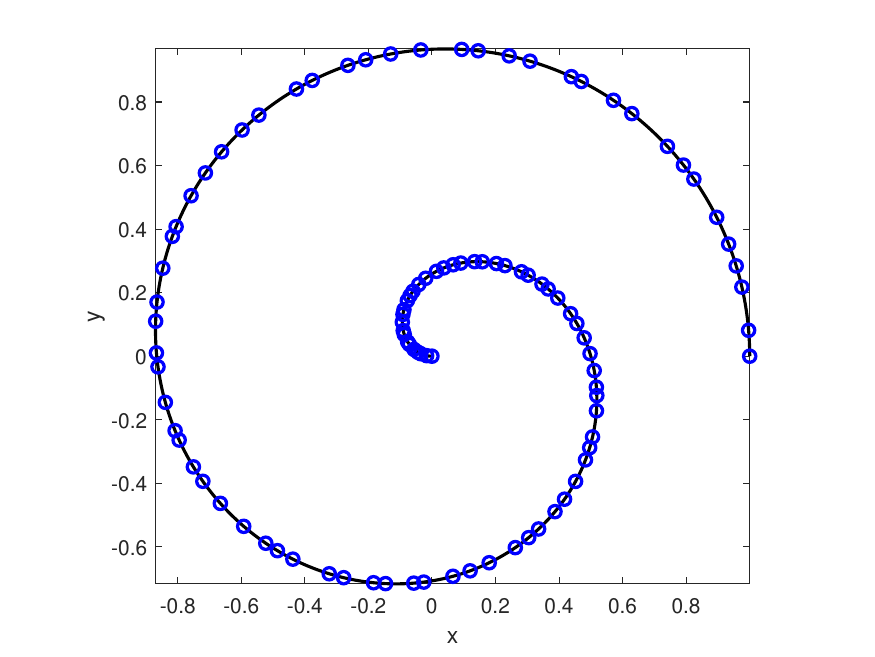}}
 \subfigure[Final mesh: $\mathbb{M}=\mathbb{I}$]{
 \includegraphics[width=0.31\textwidth, trim=20 0 20 20,clip]{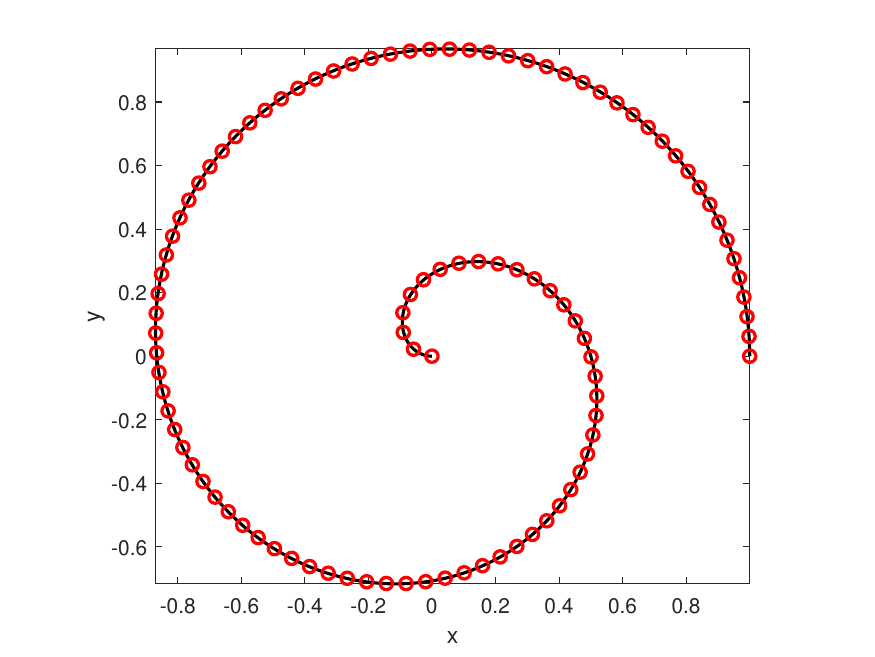}}
 \subfigure[Final mesh: $\mathbb{M} = \bar{k}\,\mathbb{I}$]{
 \includegraphics[width=0.31\textwidth, trim=20 0 20 20,clip]{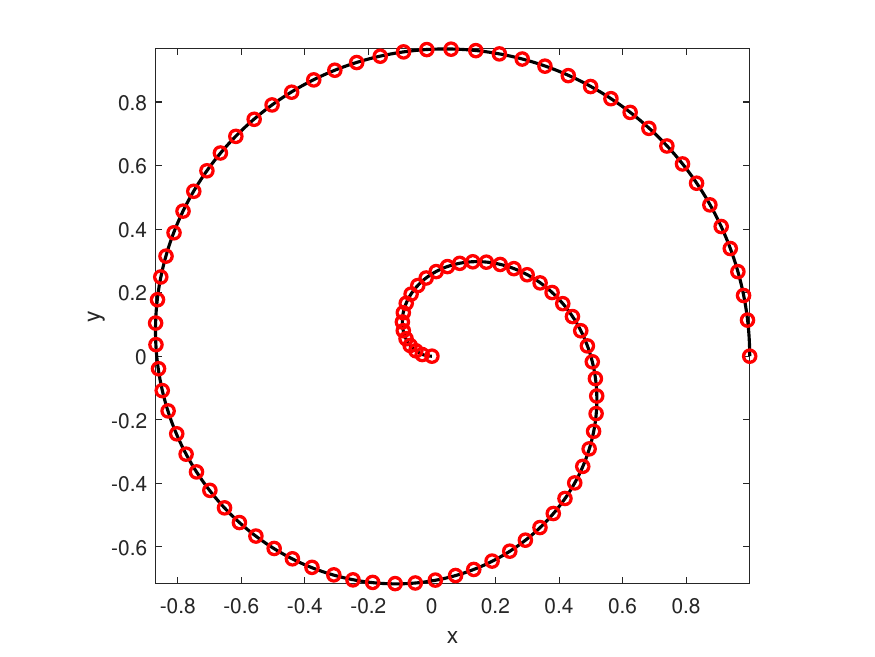}}
 \caption{{\bf Example} \ref{ex:rose}.
 Meshes ($N=100$) of the rose curve with $r = \frac{1}{6}$ and $c = 3$ are obtained by the unifying MMPDE method.  }
 \label{Fig:rose6}
 \end{figure}

\begin{figure}[H]
 \centering
 \subfigure[Initial mesh]{
 \includegraphics[width=0.31\textwidth, trim=20 0 20 20,clip]{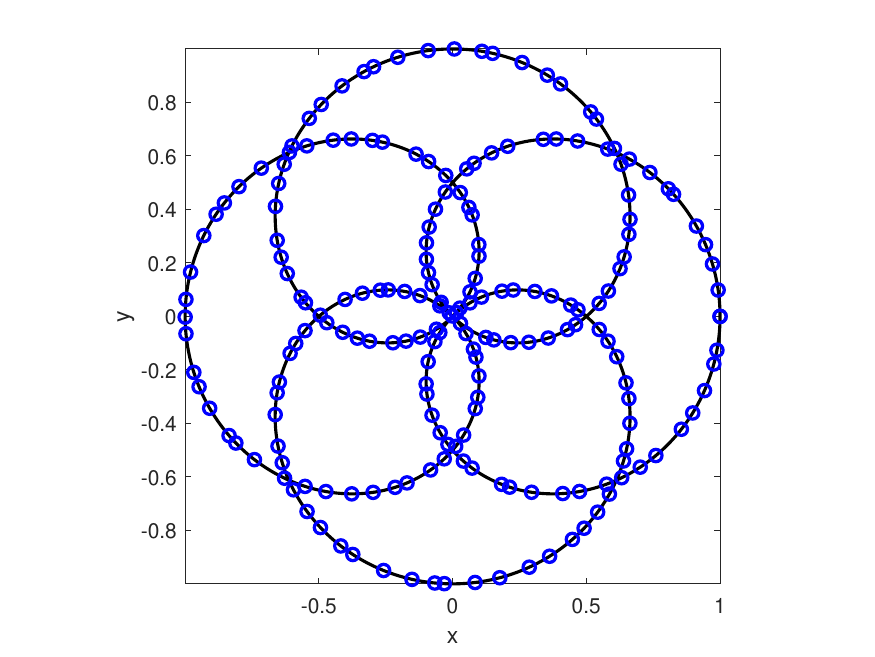}}
 \subfigure[Final mesh: $\mathbb{M}=\mathbb{I}$]{
 \includegraphics[width=0.31\textwidth, trim=20 0 20 20,clip]{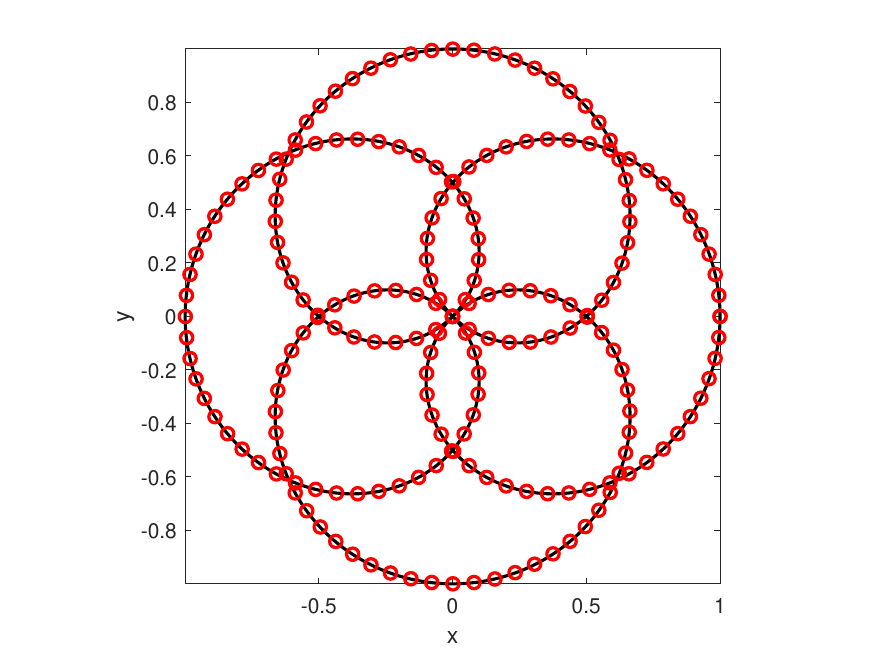}}
 \subfigure[Final mesh: $\mathbb{M} = \bar{k}\,\mathbb{I}$]{
 \includegraphics[width=0.31\textwidth, trim=20 0 20 20,clip]{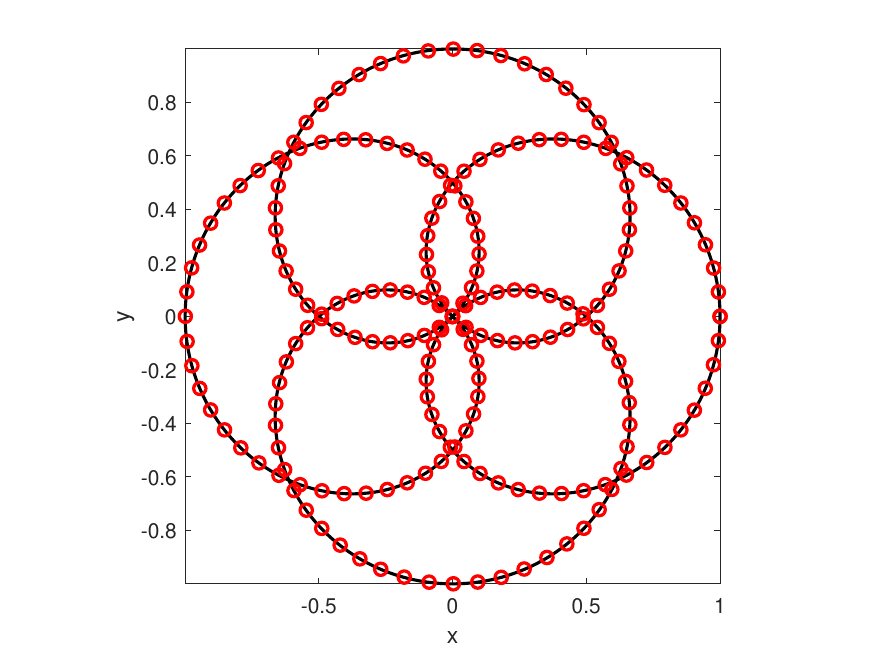}}
 \caption{{\bf Example} \ref{ex:rose}.
 Meshes ($N=200$) of the rose curve with $r = \frac{2}{3}$ and $c = 6$ are obtained by the unifying MMPDE method. }
 \label{Fig:rose3}
 \end{figure}

\begin{figure}[H]
 \centering
 \subfigure[Initial mesh]{
 \includegraphics[width=0.31\textwidth, trim=20 0 20 20,clip]{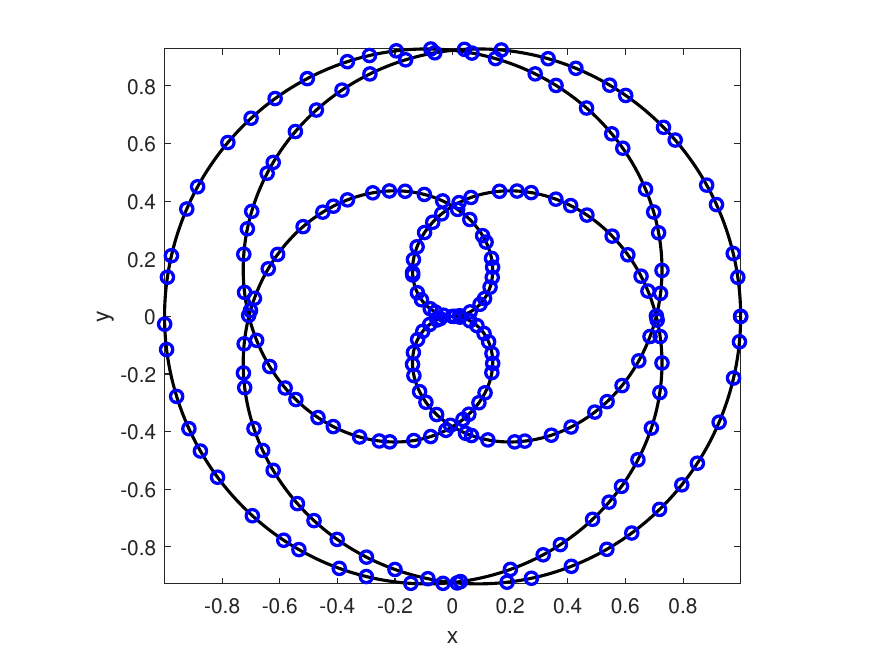}}
 \subfigure[Final mesh: $\mathbb{M}=\mathbb{I}$]{
 \includegraphics[width=0.31\textwidth, trim=20 0 20 20,clip]{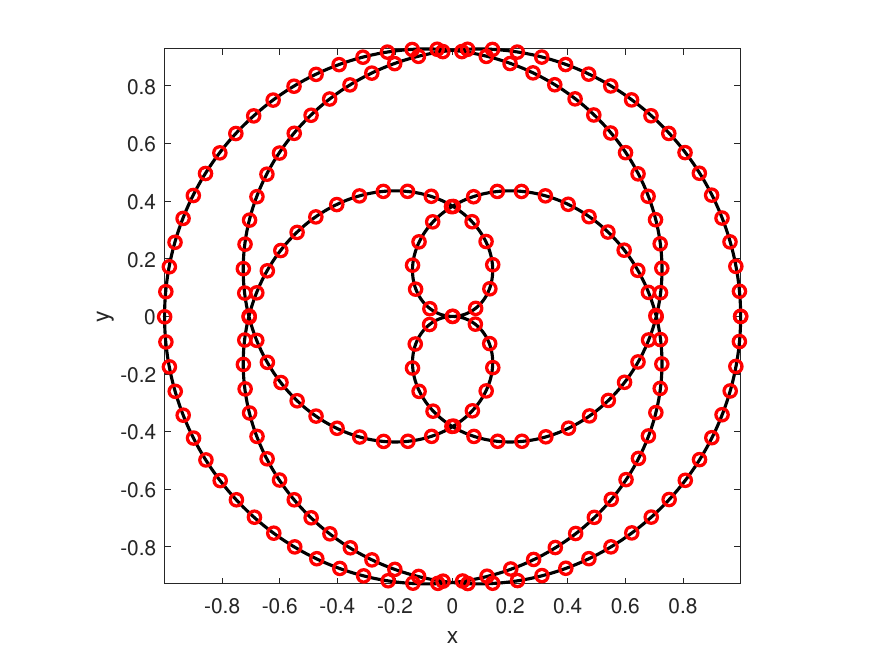}}
  \subfigure[Final mesh: $\mathbb{M} = \bar{k}\,\mathbb{I}$]{
 \includegraphics[width=0.31\textwidth, trim=20 0 20 20,clip]{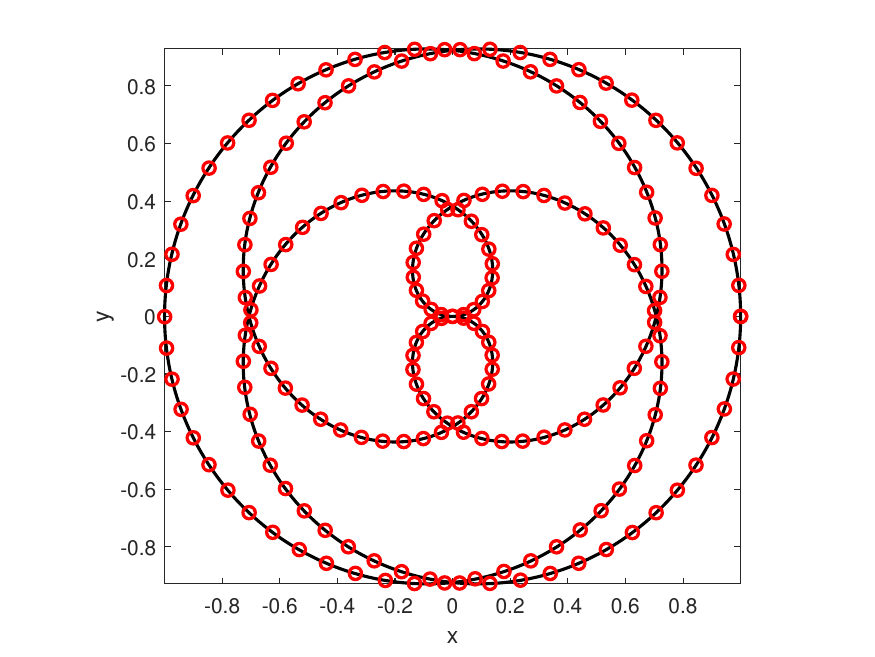}}
 \caption{{\bf Example} \ref{ex:rose}.
 Meshes ($N=200$) of the rose curve with $r = \frac{1}{4}$ and $ c = 8$ are obtained by the unifying MMPDE method. }
 \label{Fig:rose2}
 \end{figure}

\begin{figure}[H]
 \centering
 \subfigure[Initial mesh]{
 \includegraphics[width=0.31\textwidth, trim=20 0 20 20,clip]{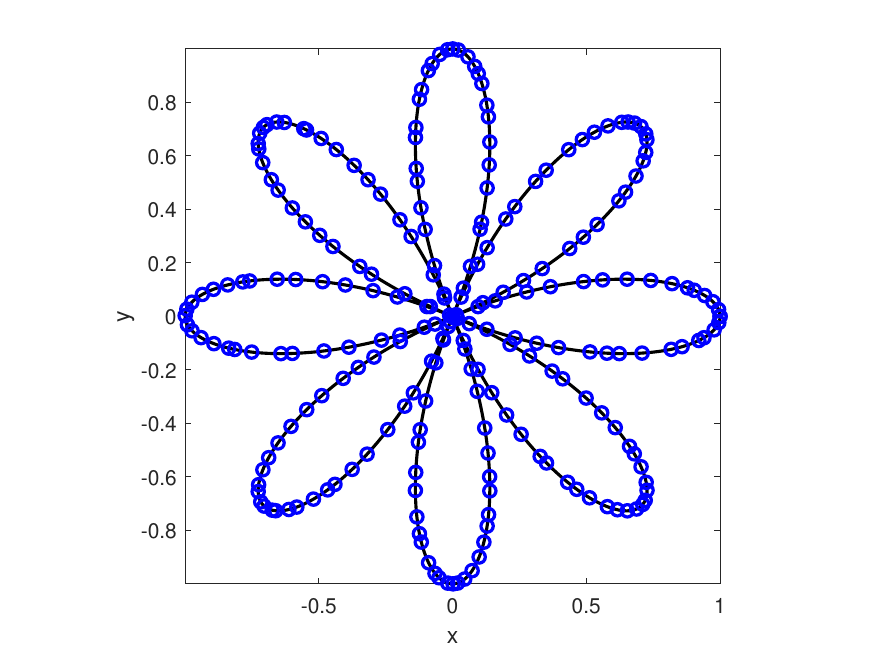}}
 \subfigure[Final mesh: $\mathbb{M}=\mathbb{I}$]{
 \includegraphics[width=0.31\textwidth, trim=20 0 20 20,clip]{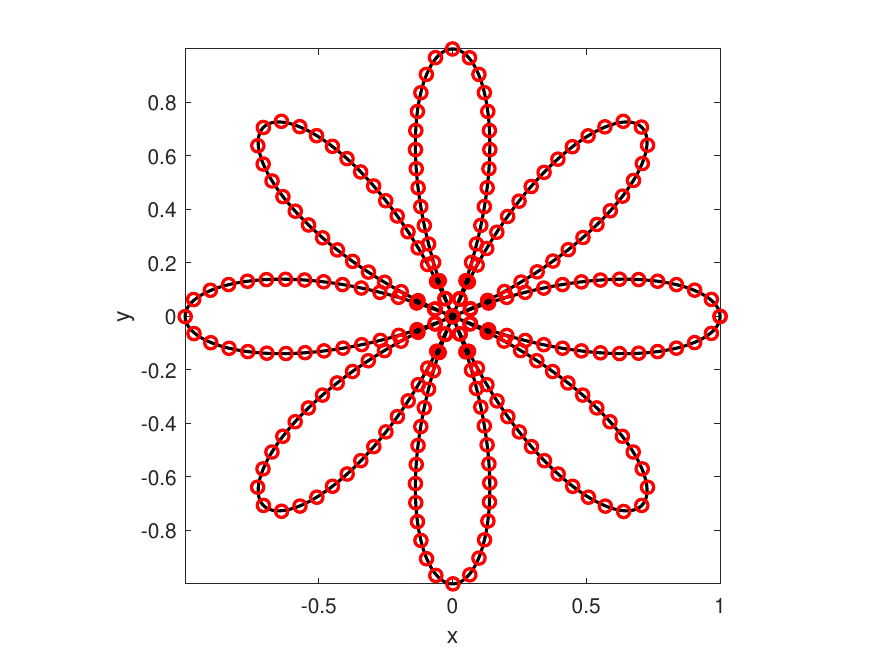}}
 \subfigure[Final mesh: $\mathbb{M} = \bar{k}\,\mathbb{I}$]{
 \includegraphics[width=0.31\textwidth, trim=20 0 20 20,clip]{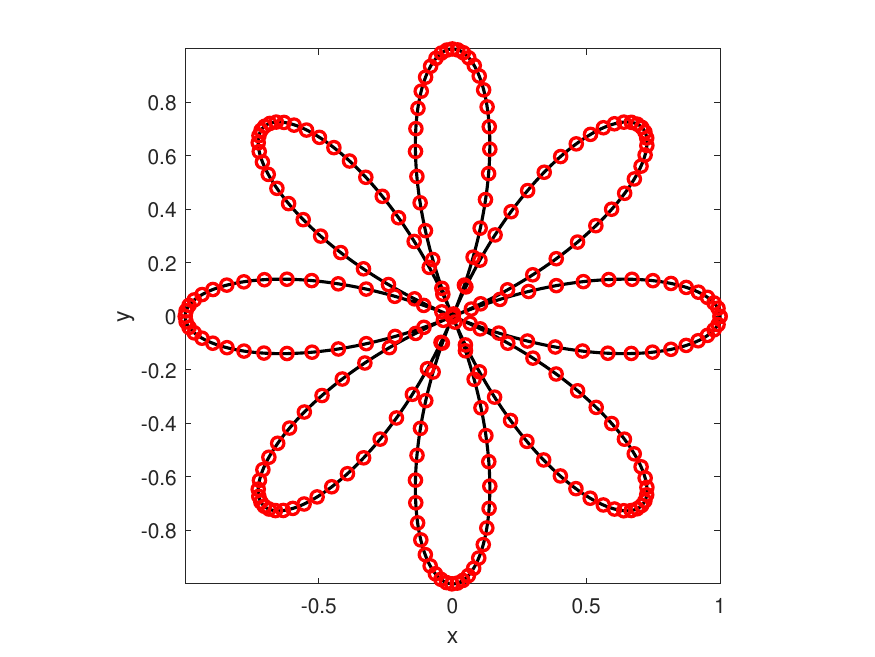}}
 \caption{{\bf Example} \ref{ex:rose}.
 Meshes ($N=240$) of the rose curve with $r = 4$ and $c = 2$ (an eight-petal rose) are obtained by the unifying MMPDE method.
 }
 \label{Fig:rose1}
 \end{figure}

 \begin{example}\label{ex:MexicanCap}
 (3D  Mexican cap curve)
\end{example}
In this example, we consider a 3D Mexican cap curve that has the parametric representation as
\begin{equation}
x(s) = e^{0.1s}\cos(10s), \quad
 y(s) = e^{0.1s}\sin(10s),\quad z(s) = s, \quad s\in [-12,12].
 \end{equation}
A randomly perturbed nonuniform initial mesh with $N=300$ is given in Fig.~\ref{Fig:MexicanCap}.
The final meshes for both the Euclidean metric $\mathbb{M} = \mathbb{I}$ and curvature-based metric $\mathbb{M} = \bar{k}\,\mathbb{I}$
are also shown in Fig.~\ref{Fig:MexicanCap}.
We can see that the unifying moving mesh method works well too for this 3D curve and the final meshes are evenly distributed
along the curve. Notice that the meshes for the Euclidean and curvature-based metric tensors
are almost identical since the curvature of the curve is constant.

 \begin{figure}[H]
 \centering
 \subfigure[Initial mesh]{
 \includegraphics[width=0.31\textwidth, trim=120 0 120 20,clip]{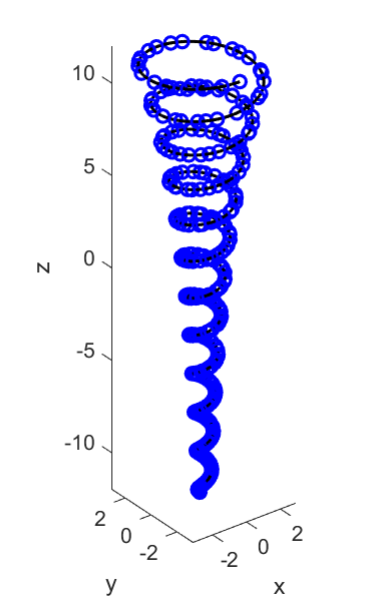}}
 \subfigure[Final mesh: $\mathbb{M} = \mathbb{I}$]{
 \includegraphics[width=0.31\textwidth, trim=120 0 120 20,clip]{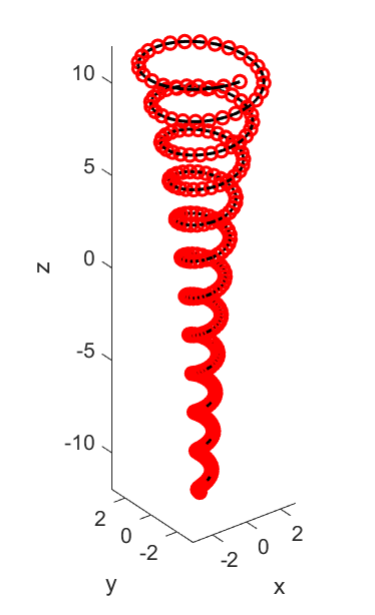}}
 \subfigure[Final mesh: $\mathbb{M} = \bar{k}\,\mathbb{I}$]{
 \includegraphics[width=0.31\textwidth, trim=120 0 120 20,clip]{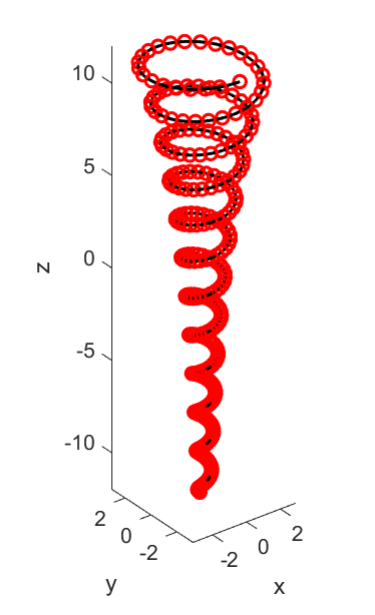}}
  \subfigure[top view of (a) ]{
 \includegraphics[width=0.31\textwidth, trim=20 0 20 20,clip]{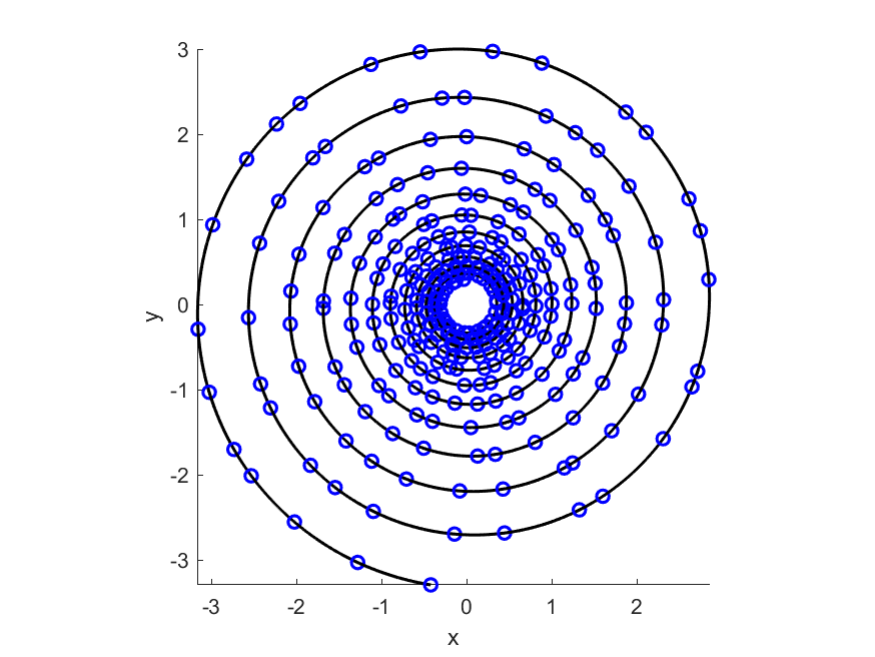}}
 \subfigure[top view of (b)]{
 \includegraphics[width=0.31\textwidth, trim=20 0 20 20,clip]{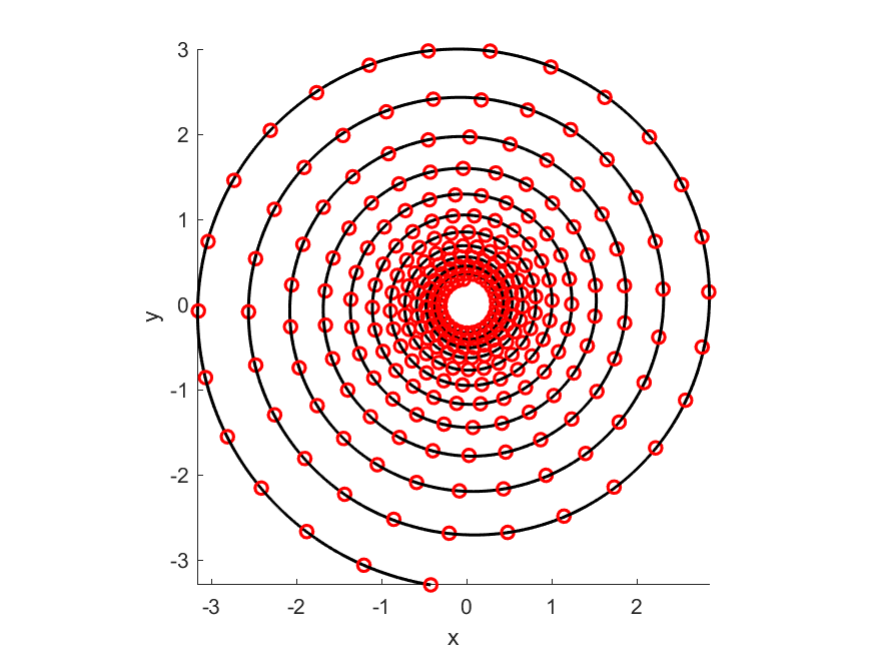}}
 \subfigure[top view of (c)]{
 \includegraphics[width=0.31\textwidth, trim=20 0 20 20,clip]{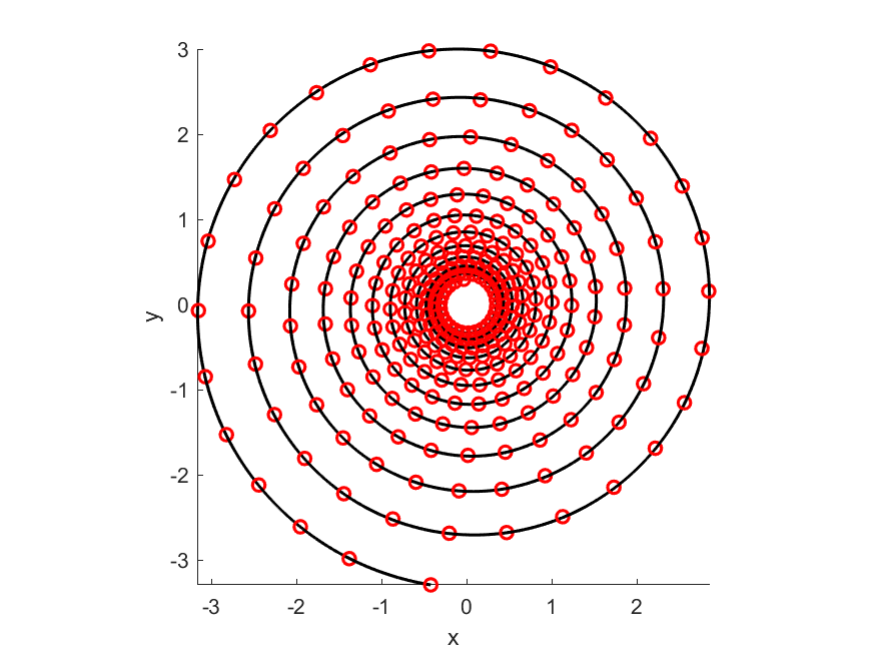}}
 \caption{{\bf Example} \ref{ex:MexicanCap}.
 Meshes ($N=300$) of the 3D Mexican cap curve are obtained by the unifying MMPDE method.
  }
 \label{Fig:MexicanCap}
 \end{figure}

  \begin{example}\label{ex:TorusCurve}
  (3D Torus curve)
 \end{example}
 In this example, we consider a 3D Torus curve that can be expressed parametrically by
 \begin{equation}
 \begin{cases}
 x(s) = (3+\cos(\sqrt{2}s))\cos(s), \\
 y(s) = \cos(\sqrt{2}s),\\
 z(s) = (3+\cos(\sqrt{2}s))\sin(s),
 \end{cases}
 \quad s\in [0,40\pi] .
 \end{equation}
 A randomly perturbed nonuniform initial mesh ($N=180$) is given in Fig.~\ref{Fig:TorusCurve}(a).
 The final meshes for Euclidean metric $\mathbb{M} = \mathbb{I}$ and curvature-based metric $\mathbb{M} = \bar{k}\,\mathbb{I}$
 are shown in Fig.~\ref{Fig:TorusCurve}(b) and (c), respectively. The final meshes are evenly distributed along the curves. (They look like
 to have higher mesh concentration at corners in the top and front view because of the limitation of viewing angle.)
 Moreover, the final meshes are almost identical because the curvature is constant along the curve.
 This example shows that the unifying moving mesh method works well for complex 3D curves.

  \begin{figure}[H]
  \centering
  \subfigure[Initial mesh]{
  \includegraphics[width=0.31\textwidth, trim=80 0 90 30,clip]{TorusCurve_Nx180_initial-eps-converted-to.pdf}}
  \subfigure[Final mesh: $\mathbb{M}_K = \mathbb{I}$]{
  \includegraphics[width=0.31\textwidth, trim=80 0 90 30,clip]{TorusCurve_Nx180_M1-eps-converted-to.pdf}}
  \subfigure[Final mesh: $\mathbb{M} = \bar{k}\,\mathbb{I}$]{
  \includegraphics[width=0.31\textwidth, trim=80 0 90 30,clip]{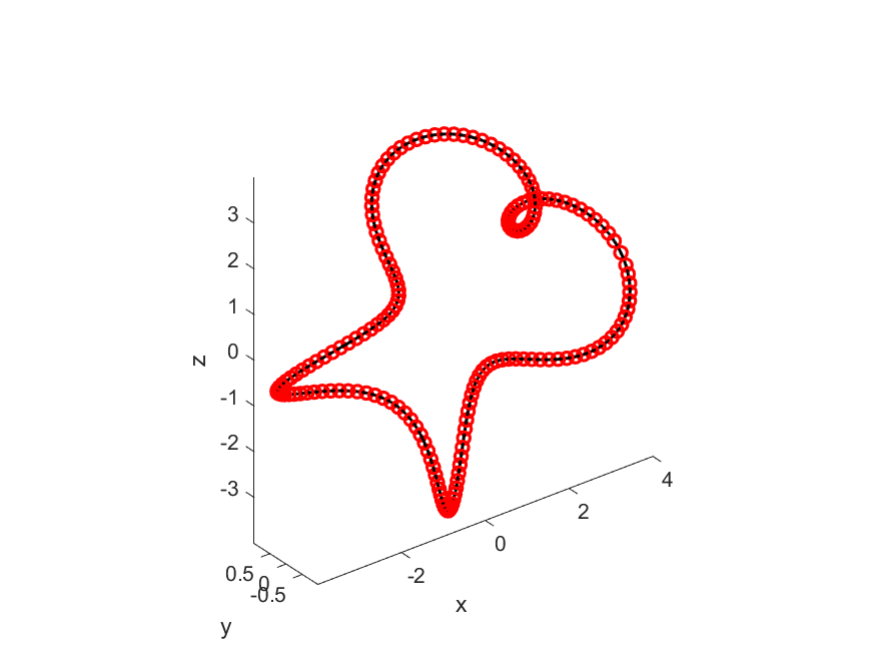}}
  \subfigure[front view of (a)]{
  \includegraphics[width=0.31\textwidth, trim=50 0 50 20,clip]{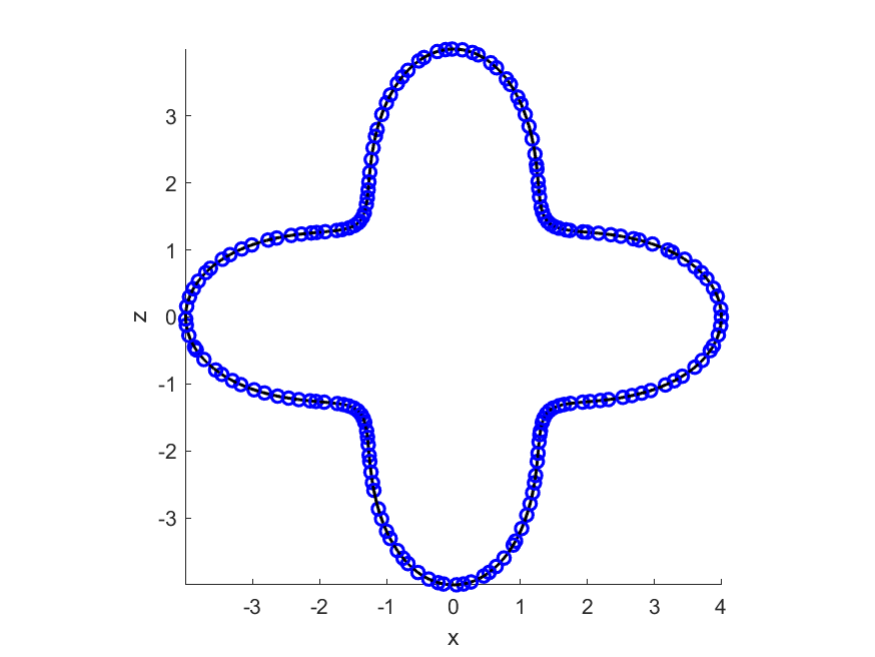}}
 \subfigure[front view of (b)]{
 \includegraphics[width=0.31\textwidth, trim=50 0 50 20,clip]{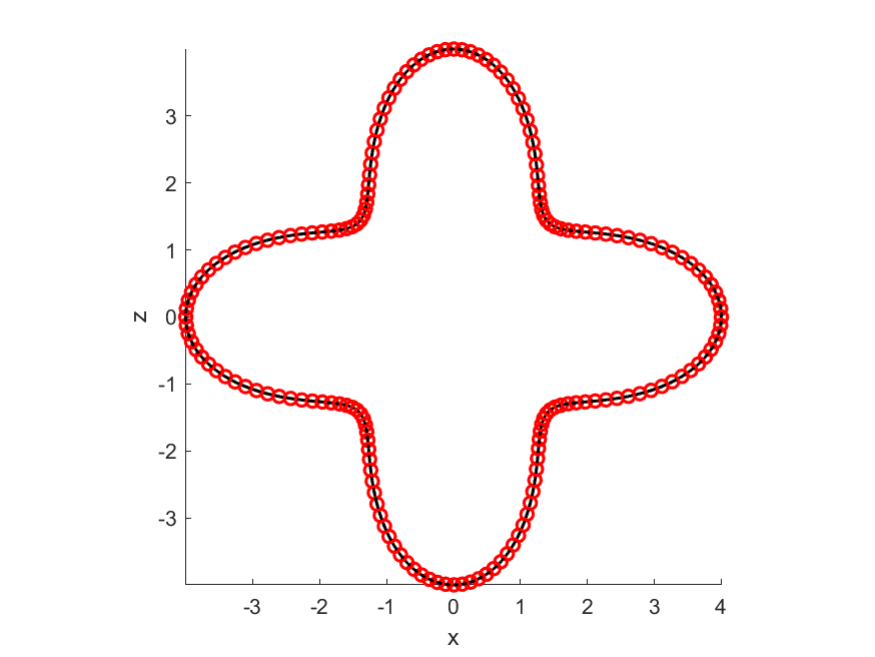}}
  \subfigure[front view of (c)]{
  \includegraphics[width=0.31\textwidth, trim=50 0 50 20,clip]{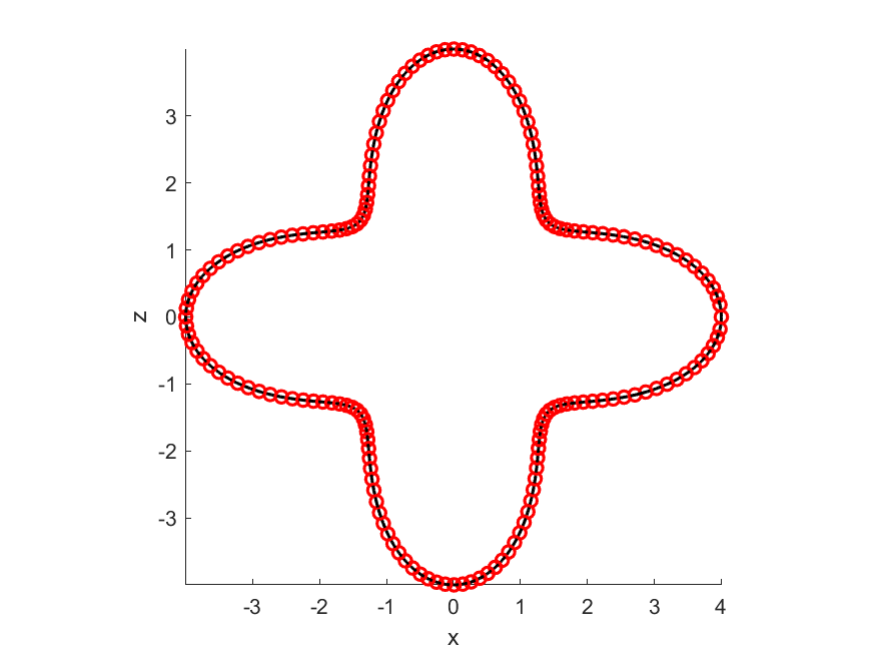}}
  \subfigure[top view of (a)]{
  \includegraphics[width=0.31\textwidth, trim=20 80 20 70,clip]{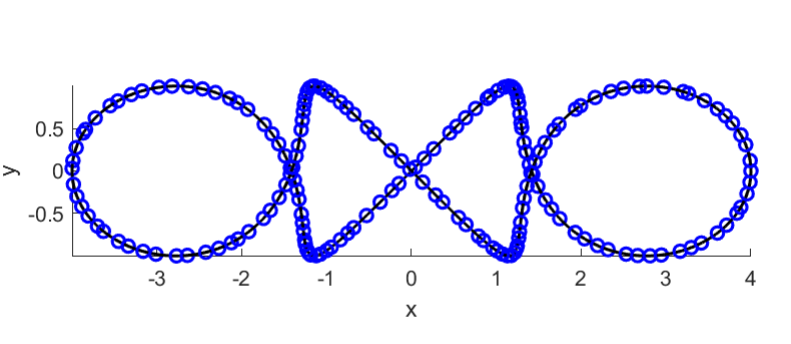}}
 \subfigure[top view of (b)]{
 \includegraphics[width=0.31\textwidth, trim=20 80 20 70,clip]{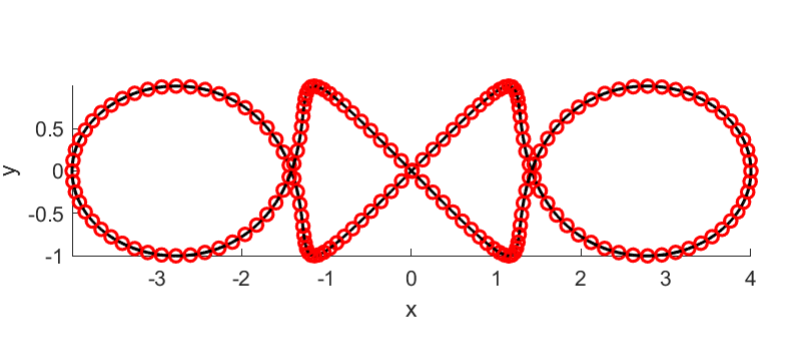}}
  \subfigure[top view of (c)]{
  \includegraphics[width=0.31\textwidth, trim=20 80 20 70,clip]{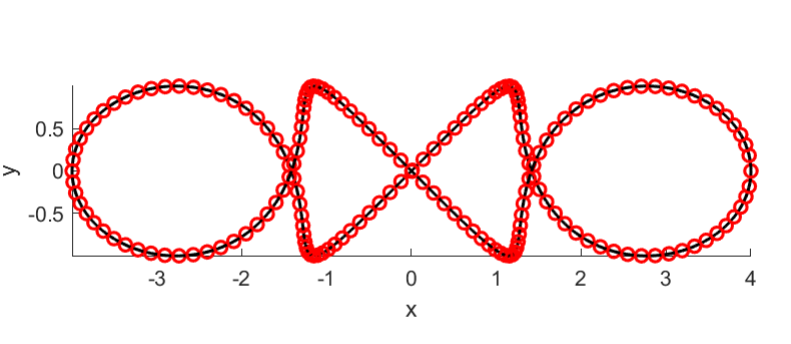}}
  \caption{{\bf Example} \ref{ex:TorusCurve}.
  Meshes ($N=180$) of the 3D torus curve are obtained by the unifying MMPDE method.
   }
  \label{Fig:TorusCurve}
  \end{figure}


 \begin{example}\label{ex:Hyperbaloid}
 (3D Hyperboloid surface)
\end{example}
In this example, we consider a 3D hyperboloid surface having the parametric representation as
\begin{equation}
\begin{cases}
x(s,\zeta) = \sqrt{(1+\zeta^2)}\cos(s), \\
y(s,\zeta) = \sqrt{(1+\zeta^2)}\sin(s),\\
z(s,\zeta) = \zeta,
\end{cases}
\quad s\in [0,2\pi],~\zeta\in[-2,2] .
\end{equation}
A randomly perturbed nonuniform initial mesh of $N=3872$ is given in Fig.~\ref{Fig:Hyperbaloid}(a).

The final meshes with the Euclidean metric $\mathbb{M} = \mathbb{I}$ and curvature-based metric $\mathbb{M} = \bar{k}\,\mathbb{I}$
are shown in Figs.~\ref{Fig:Hyperbaloid}(b) and (c), respectively.
We can see that the unifying moving mesh method works well for surface mesh movement and
provides an effective control of mesh concentration through the metric tensor.
More specifically, the final mesh with the Euclidean metric $\mathbb{M} = \mathbb{I}$ is evenly distributed while
the mesh with the curvature-based metric tensor is visibly coarser at the middle, top, and bottom sections
where the curvature is smaller.

 \begin{figure}[H]
 \centering
 \subfigure[Initial mesh]{
 \includegraphics[width=0.31\textwidth, trim=60 0 60 20,clip]{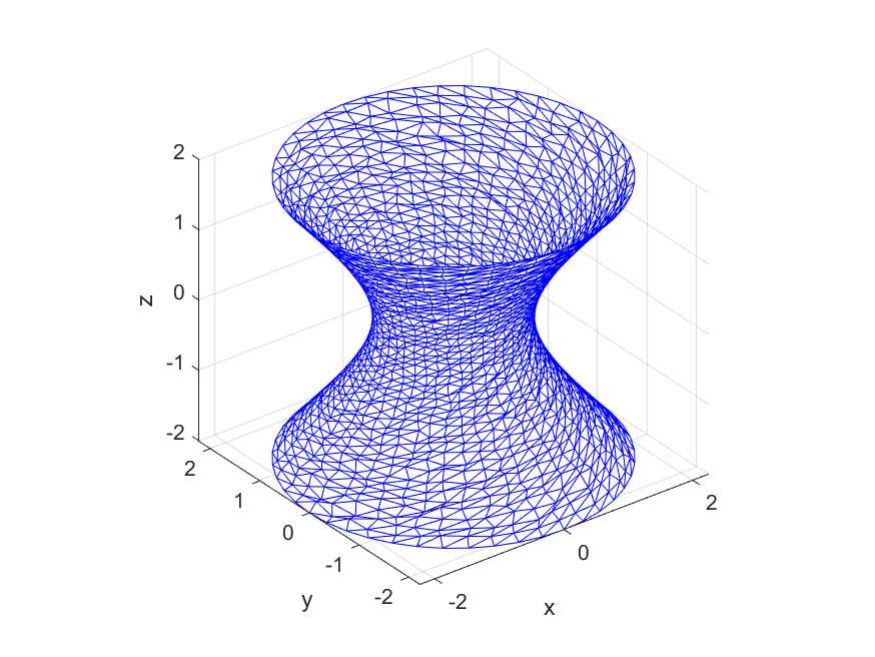}}
 \subfigure[Final mesh: $\mathbb{M}_K = \mathbb{I}$]{
 \includegraphics[width=0.31\textwidth, trim=60 0 60 20,clip]{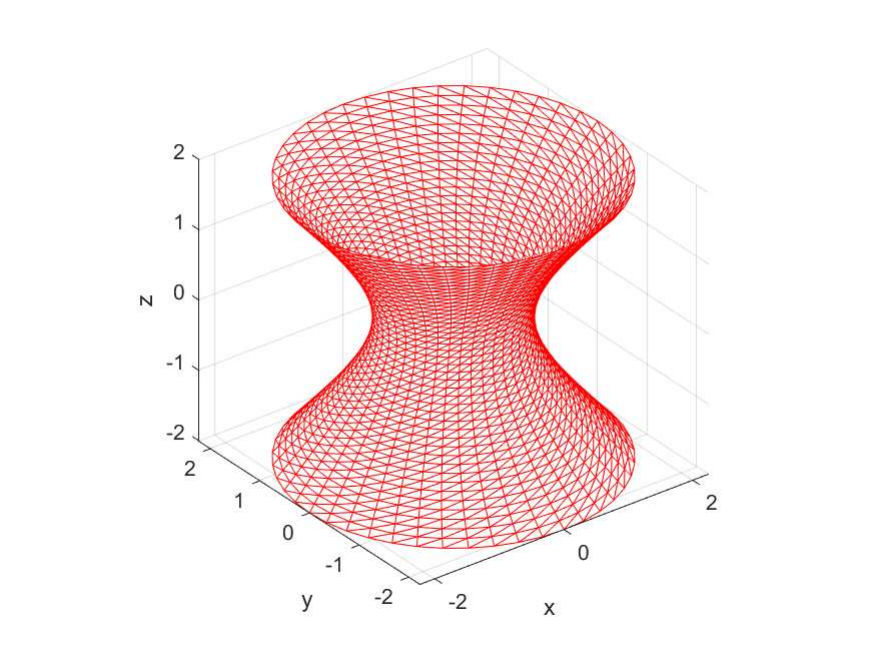}}
 \subfigure[Final mesh: $\mathbb{M} = \bar{k}\,\mathbb{I}$]{
 \includegraphics[width=0.31\textwidth, trim=60 0 60 20,clip]{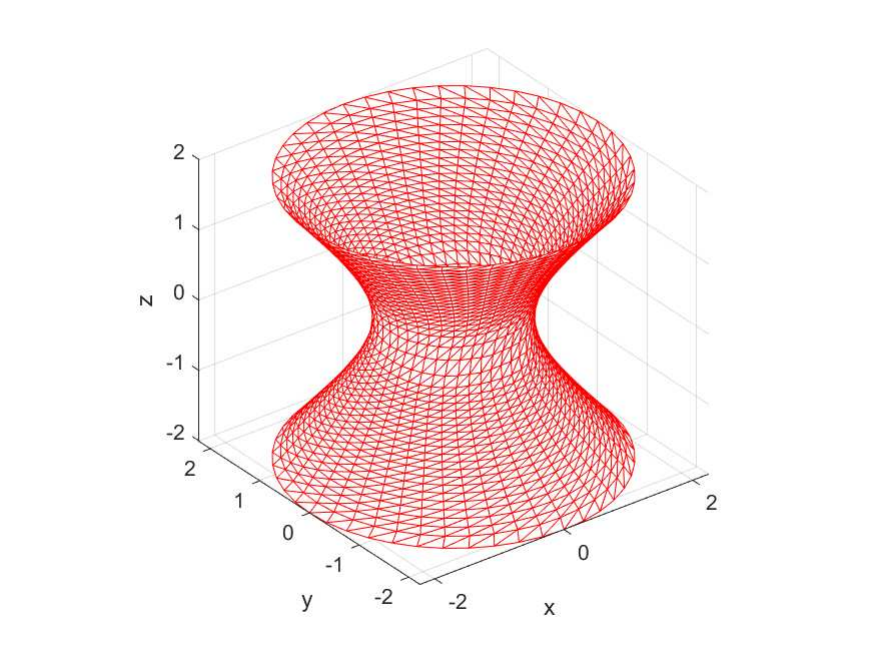}}
  \subfigure[front view of (a)]{
 \includegraphics[width=0.31\textwidth, trim=40 0 40 20,clip]{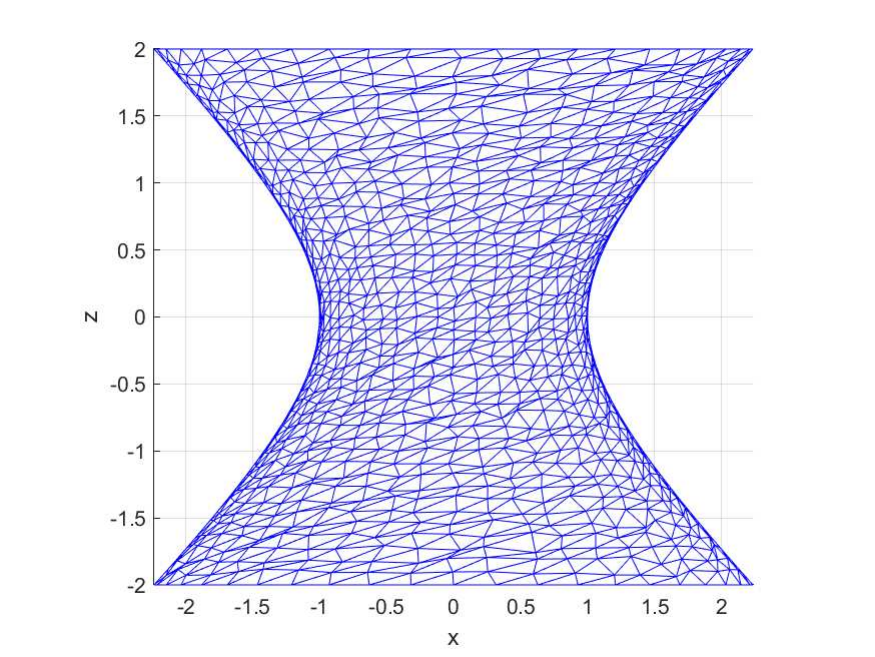}}
 \subfigure[front view of (b)]{
 \includegraphics[width=0.31\textwidth, trim=40 0 40 20,clip]{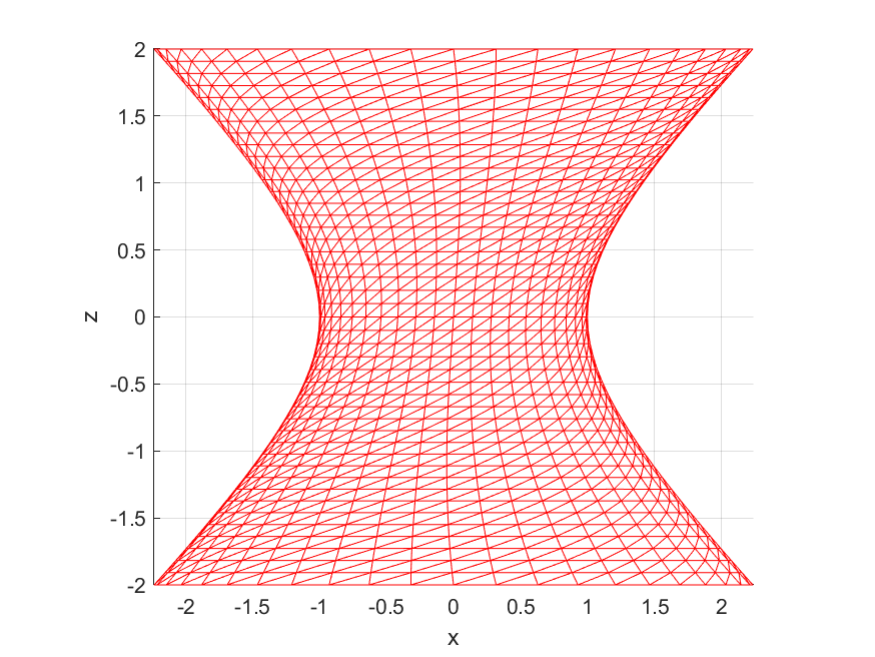}}
 \subfigure[front view of (c)]{
 \includegraphics[width=0.31\textwidth, trim=40 0 40 20,clip]{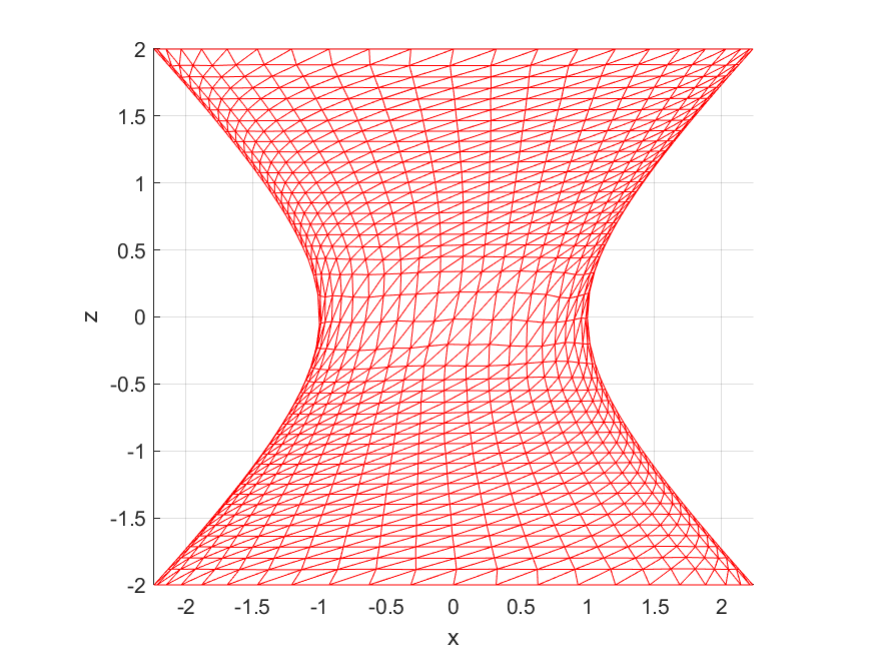}}
 \caption{{\bf Example} \ref{ex:Hyperbaloid}.
 Meshes of $N=3872$ for a 3D hyperboloid surface are obtained by the unifying MMPDE method.
 }
 \label{Fig:Hyperbaloid}
 \end{figure}

  \begin{example}\label{ex:Cavatappi}
  (3D Cavatappi surface)
\end{example}
In this example, we consider the 3D Cavatappi surface that can be expressed parametrically as
\begin{equation}
\begin{cases}
x(s,\zeta) = 3+2\cos(\frac{\pi}{35}s)
+0.1\cos(\frac{2\pi}{7}s)\cos(\frac{\pi}{30}\zeta), \\
y(s,\zeta) =  3+2\cos(\frac{\pi}{35}s)
+0.1\cos(\frac{2\pi}{7}s)\sin(\frac{\pi}{30}\zeta),\\
z(s,\zeta) =  3+2\sin(\frac{\pi}{35}s)
+0.1\sin(\frac{2\pi}{7}s) + \frac{1}{6}\zeta ,
\end{cases}
~~ s\in [0,70],~\zeta\in[0,150] .
\end{equation}
A randomly perturbed nonuniform initial mesh of $N=21000$ is given in Fig.~\ref{Fig:Cavatappi}(a).
The final meshes with the Euclidean and the curvature-based metric tensors are shown in Figs.~\ref{Fig:Cavatappi}(b) and (c), respectively.
They are almost identical since the curvature of the surface is constant and evenly distributed as expected for the metric tensors.
This example confirms that the unifying moving mesh method can work well for complex surfaces.

 \begin{figure}[H]
 \centering
 \subfigure[Initial mesh]{
 \includegraphics[width=0.31\textwidth, trim=125 0 130 20,clip]{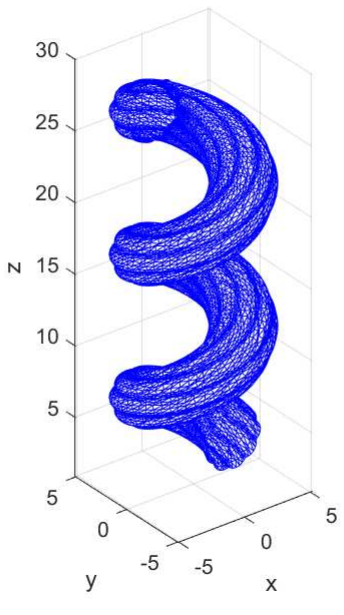}}
 \subfigure[Final mesh: $\mathbb{M} = \mathbb{I}$]{
 \includegraphics[width=0.31\textwidth, trim=125 0 130 20,clip]{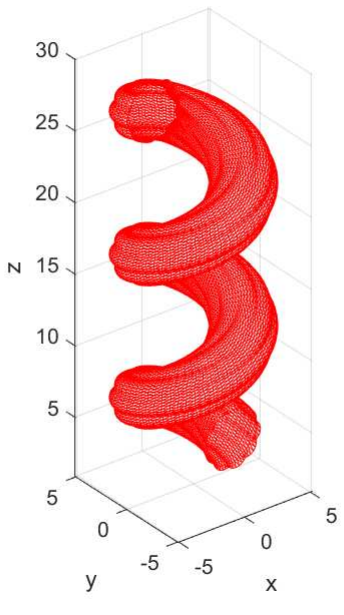}}
 \subfigure[Final mesh: $\mathbb{M} = \bar{k}\,\mathbb{I}$]{
 \includegraphics[width=0.31\textwidth, trim=125 0 130 20,clip]{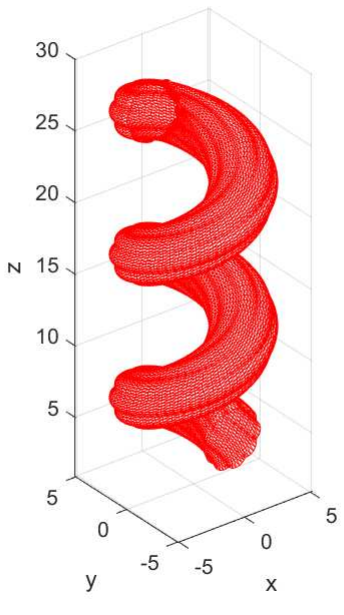}}
 \subfigure[top view of (a)]{
 \includegraphics[width=0.31\textwidth, trim=20 0 30 20,clip]{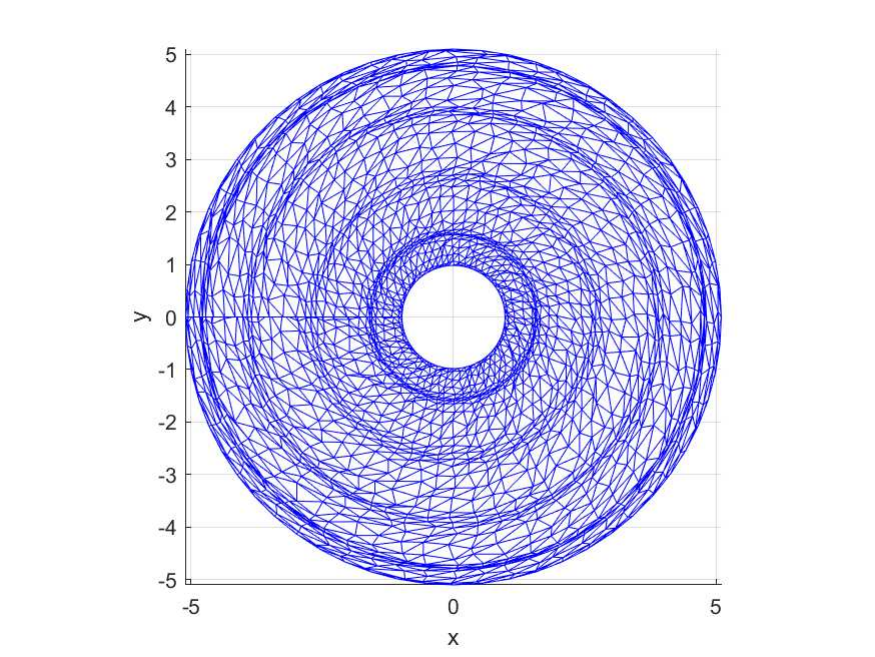}}
 \subfigure[top view of (b)]{
 \includegraphics[width=0.31\textwidth, trim=2 0 30 20,clip]{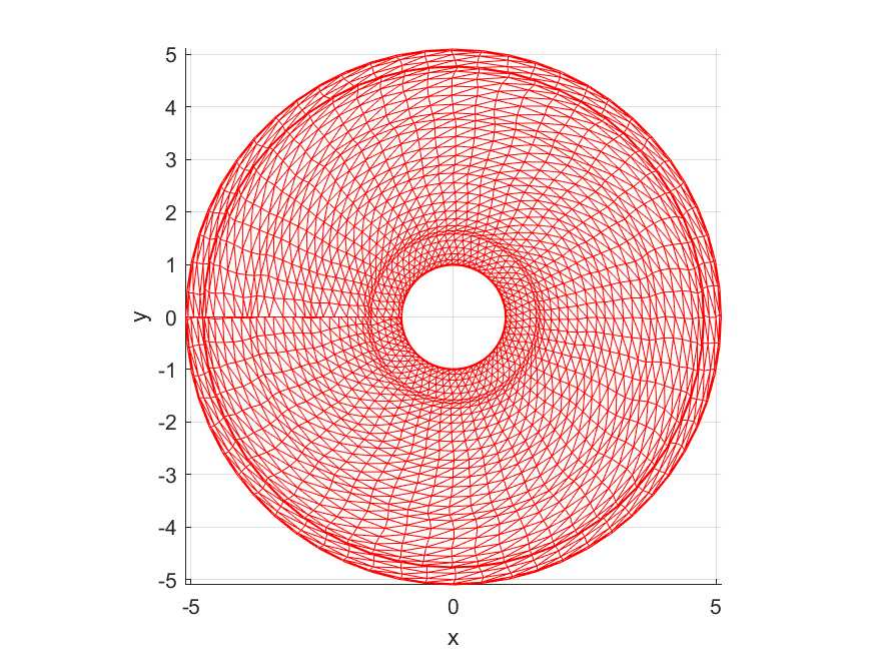}}
  \subfigure[top view of (c)]{
 \includegraphics[width=0.31\textwidth, trim=2 0 30 20,clip]{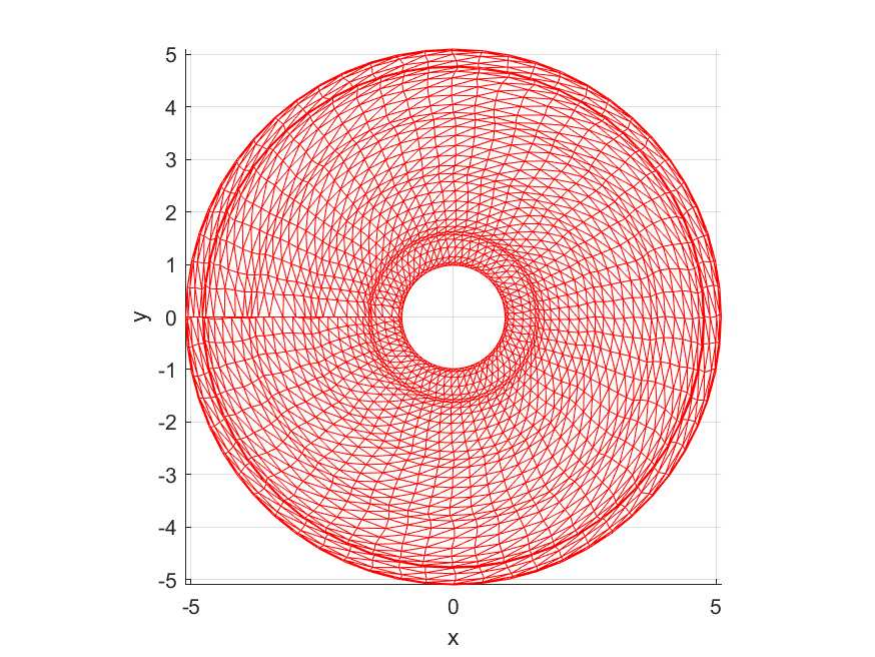}}
 \caption{{\bf Example} \ref{ex:Cavatappi}.
 Meshes of $N=21000$ for the 3D Cavatappi surface are obtained by the unifying MMPDE method.
 }
 \label{Fig:Cavatappi}
 \end{figure}

\section{Conclusions and remarks}
\label{sec:conclusion}

In the previous sections we have presented a unifying moving mesh method for a general $m$-dimensional geometry $S$
that can be a curve, a surface, or a domain in $\R^d$ ($d \ge 1$ and $1\le m \le d$). The method does not require the availability of
an analytical parametric representation of $S$.
Serval properties of edge matrices and affine mappings of general $m$-simplexes have been studied.
Based on these properties, we have established the mathematical characterization of the $m$-simplicial nonuniform meshes in the unifying form.
The equidistribution and alignment conditions are used to characterize the size, shape, and orientation of the $m$-simplicial mesh and develop an energy function for mesh optimization.
The unifying moving mesh equation is defined based on the MMPDE approach, and
suitable projection of the nodal mesh velocities is employed to ensure the mesh nodes stay on $S$.
The analytical expression for the mesh velocities has been derived
in a compact matrix form by using scalar-by-matrix differentiation.
The mesh nonsingularity in adaptation and generation of our unifying moving mesh method has been proved.

The numerical results for curves ($m=1$) and surfaces ($m=2$) in $\mathbb{R}^2$ and $\mathbb{R}^3$ were presented to verify the ability of the unifying moving mesh method to move and concentrate the mesh points.
It is worth emphasizing that our unifying moving mesh method does not require the availability of an analytical parametric representation of the underlying curve/surface. Numerical approximation of
the needed information on the normal/tangent vector for $S$ (or mean curvature for the curvature-based metric tensor)
can be obtained using the initial or current mesh representation of $S$.

It is worth pointing out that we only consider both the Euclidean and curvature-based metric tensors
in the unifying moving mesh method in this work.
Studies of other Riemannian metric tensors of mesh adaptation and generation for general $m$-dimensional geometry $S$ in $\R^d$
will be an interesting research topic in the near future.
Other future work will include application of the unifying moving mesh method to the numerical solution of
geometric PDEs and PDEs defined on a general $m$-dimensional geometry $S$ in $\R^d$.

\section*{Acknowledgements}
Min Zhang was supported in part by the National Natural Science Foundation of China (Grant Number: 12301493).
The computational resources were partially supported by High-performance Computing Platform of Peking University.

\end{document}